\documentclass{amsart}
\usepackage{mathtools}
\usepackage{biblatex}
\DeclareNameAlias{author}{last-first}

\usepackage{scalerel}
\usepackage{tikz}
\usetikzlibrary{svg.path}

\definecolor{orcidlogocol}{HTML}{A6CE39}
\tikzset{
  orcidlogo/.pic={
    \fill[orcidlogocol] svg{M256,128c0,70.7-57.3,128-128,128C57.3,256,0,198.7,0,128C0,57.3,57.3,0,128,0C198.7,0,256,57.3,256,128z};
    \fill[white] svg{M86.3,186.2H70.9V79.1h15.4v48.4V186.2z}
                 svg{M108.9,79.1h41.6c39.6,0,57,28.3,57,53.6c0,27.5-21.5,53.6-56.8,53.6h-41.8V79.1z M124.3,172.4h24.5c34.9,0,42.9-26.5,42.9-39.7c0-21.5-13.7-39.7-43.7-39.7h-23.7V172.4z}
                 svg{M88.7,56.8c0,5.5-4.5,10.1-10.1,10.1c-5.6,0-10.1-4.6-10.1-10.1c0-5.6,4.5-10.1,10.1-10.1C84.2,46.7,88.7,51.3,88.7,56.8z};
  }
}

\newcommand\orcidicon[1]{\href{https://orcid.org/#1}{\mbox{\scalerel*{
\begin{tikzpicture}[yscale=-1,transform shape]
\pic{orcidlogo};
\end{tikzpicture}
}{|}}}}

\usepackage{tikz-cd}
\usepackage{xcolor}
\usepackage{hyperref}
\hypersetup{
    colorlinks=true,
    linkcolor={red!50!black},
    citecolor={blue!50!black},
    urlcolor={blue!20!black}
}
\usepackage{thmtools}
\usepackage[noabbrev,capitalize]{cleveref}

\declaretheorem[name=Definition,style=definition,numberwithin=section,qed={\hfill $\triangle$}]{definition}
\declaretheorem[name=Example,style=remark,numberwithin=section,qed={\hfill $\triangle$}]{example}
\declaretheorem[name=Remark,style=remark,numberwithin=section,qed={\hfill $\triangle$}]{remark}

\crefname{observation}{observation}{observations}
\Crefname{observation}{Observation}{Observations}

\crefname{conjecture}{conjecture}{conjectures}
\Crefname{conjecture}{Conjecture}{Conjectures}

\crefname{assumption}{assumption}{assumptions}
\Crefname{assumption}{Assumption}{Assumptions}

\begin{document}

\def\im{\mathrm{Im}\, }
\def\re{\mathrm{Re}\, }
\def\diam{\mathrm{diam}\, }
\def\Hess{\mathrm{Hess}\, }
\def\rank{\mathrm{rank }\,  }
\def\rg{\mathrm{rg }\,  }
\def\d{\mathrm{d}}
\def\p{\partial }
\def\pr{\mathrm{pr}}
\def\D{\mathrm{D}}
\def\id{\mathrm{id}}
\def\Id{\mathrm{Id}}
\def\e{\epsilon}
\def\C{\mathbb{C}}
\def\Z{\mathbb{Z}}
\def\Q{\mathbb{Q}}
\def\N{\mathbb{N}}
\def\R{\mathbb{R}}
\def\K{\mathbb{K}}
\def\T{\mathbb{T}}
\def\diag{\mathrm{diag}\, }

\newtheorem{theorem}{Theorem}[section]
\newtheorem{corollary}[theorem]{Corollary}
\newtheorem{lemma}[theorem]{Lemma}
\newtheorem{proposition}[theorem]{Proposition}
\newtheorem{prop}[theorem]{Proposition}
\newtheorem{conjecture}[theorem]{Conjecture}
\newtheorem{assumption}{Assumption}[section]

\numberwithin{equation}{section}


\makeatletter
\renewcommand*\env@matrix[1][\arraystretch]{%
  \edef\arraystretch{#1}%
  \hskip -\arraycolsep
  \let\@ifnextchar\new@ifnextchar
  \array{*\c@MaxMatrixCols c}}
\makeatother

\newsavebox{\smlmata}
\savebox{\smlmata}{${\scriptstyle\overline q(0)=\left(\protect\begin{smallmatrix}0.2\\ 0.1\protect\end{smallmatrix}\right) = \overline q(5)}$}
\newsavebox{\smlmatb}
\savebox{\smlmatb}{${\scriptstyle q =\left(\protect\begin{smallmatrix} -1&2\\3&1\protect\end{smallmatrix}\right) \overline q}$}
\newsavebox{\smlmatc}
\savebox{\smlmatc}{${\scriptstyle p(0)=\left(\protect\begin{smallmatrix}0.85\\ 1.5\protect\end{smallmatrix}\right) = p(2 \pi/3)}$}

\newcommand{\dd}{\mathrm{d}}
\DeclareRobustCommand\marksymbol[2]{\tikz[#2,scale=1.2]\pgfuseplotmark{#1};}
\DeclareRobustCommand{\CollLine}{\raisebox{2pt}{\tikz{\draw[-.,red,dash pattern={on 7pt off 2pt on 3pt off 2pt},line width = 1.5pt](0,0) -- (9mm,0);}}}
\DeclareRobustCommand{\RefLine}{\raisebox{2pt}{\tikz{\draw[-,black,dash pattern={on 7pt off 2pt},line width = 1.5pt](0,0) -- (9mm,0);}}}
\DeclareRobustCommand{\ConvLine}{\raisebox{2pt}{\tikz{\draw[-,blue,line width = 1.5pt](0,0) -- (9mm,0);}}}
\DeclareRobustCommand{\redtriangle}{\raisebox{0.5pt}{\tikz{\node[draw,scale=0.3,regular polygon, regular polygon sides=3,fill=red!10!white,rotate=0](){};}}}

\title[Lagrangian intersections and catastrophes]{Local intersections of Lagrangian manifolds correspond to catastrophe theory}

\author[Christian Offen]{Christian Offen {\protect \orcidicon{0000-0002-5940-8057}}}
\address{Department of Mathematics, Paderborn University, Warburger Straße 100, 33098 Paderborn, Germany}
\email{christian.offen@uni-paderborn.de}
\date{\today}
\keywords{Lagrangian intersections, contact equivalence, catastrophe theory}
\subjclass[2020]{53D12, 58K35, 37J20, 58K05}

\begin{abstract}
Two smooth map germs are right-equivalent if and only if they generate two Lagrangian submanifolds in a cotangent bundle which have the same contact with the zero-section.
In this paper we provide a reverse direction to this classical result of Golubitsky and Guillemin.
Two Lagrangian submanifolds of a symplectic manifold have the same contact with a third Lagrangian submanifold if and only if the intersection problems correspond to stably right equivalent map germs. We, therefore, obtain a correspondence between local Lagrangian intersection problems and catastrophe theory while the classical version only captures tangential intersections.
The correspondence is defined independently of any Lagrangian fibration of the ambient symplectic manifold, in contrast to other classical results. 
Moreover, we provide an extension of the correspondence to families of local Lagrangian intersection problems.
This gives rise to a framework which allows a natural transportation of the notions of catastrophe theory such as stability, unfolding and (uni-)versality to the geometric setting such that we obtain a classification of families of local Lagrangian intersection problems.
An application is the classification of Lagrangian boundary value problems for symplectic maps.
\end{abstract}

\maketitle

\section{Introduction}\label{sec:intro}


Local singularities of smooth, scalar-valued functions have been studied extensively under the headlines {\em catastrophe theory} and {\em singularity theory}. The local behaviour of the set of critical points of a smooth function under perturbations is related to bifurcation phenomena in dynamical systems \cite{Arnold1992,Arnold1994}. Applications to optics, economical models, and laser physics are reviewed in \cite{poston1978catastrophe}, for instance.
Thanks to foundational work by Whitney, Thom, Mather, Arnold, and others, classification results for singularities are known \cite{Arnold2012,lu1976singularity,Wassermann1974}.
Of fundamental importance for the classification results is the notion of right equivalence of function germs.

\begin{definition}[right equivalence of function germs]
Two germs of smooth functions $f,g \colon (\R^k,0) \to (\R,0)$ are {\em right equivalent} if there exists a germ of a diffeomorphism $h \colon (\R^k,0) \to (\R^k,0)$ such that $f =  g \circ h$.
\end{definition}

In \cite{golubitsky1975contact} Golubitsky and Guillemin show that the question whether two function germs are right equivalent has a geometric analogue, namely whether two Lagrangian submanifolds in a cotangent bundle have the same contact with the zero-section.

\begin{definition}[tangential contact of Lagrangian intersections (in cotangent bundles)]
Let $\Lambda$, $\Lambda'$ be two Lagrangian submanifolds in the cotangent bundle $T^\ast X$ intersecting at a point $z \in X \subset T^\ast X$ tangentially. The submanifolds $\Lambda$, $\Lambda'$ have the {\em same contact with $X$ at $z$} if and only if there exists a symplectomorphism $\Phi$ defined on an open neighbourhood $U$ of $z \in T^\ast X$ such that
\[
\Phi(\Lambda \cap U) = \Lambda' \cap U, \quad \Phi(z) = z, \quad \Phi(X\cap U) = X \cap U.
\qedhere\]
\end{definition}

Building on results by Tougeron \cite{Tougeron1968} they prove the following theorem.

\begin{theorem}[\cite{golubitsky1975contact}]\label{thm:GGinCOT}
Let $U$ be an open neighbourhood of 0 in $\R^k$.
The germs at $0$ of two smooth function $f,g \colon U \to \R$ with $f(0)=0=g(0)$, vanishing gradients at 0, i.e.\ $\d f|_0 = 0 = \d g|_0$, and vanishing Hessian matrices at $0$ (in any local coordinate system) are right equivalent if and only if the Lagrangian manifolds $\Lambda=\d f(U)$ and $\Lambda'=\d g(U)$ have the same tangential contact with the zero-section at 0 in $T^\ast U$.
\end{theorem}

In the above theorem $\d f(U)$ and $\d g(U)$ denote the image of $U \subset \R^k$ under the 1-form $\d f$ or $\d g$, respectively, where 1-forms are interpreted as maps $U \to T^\ast U$, i.e.\ as sections of the cotangent bundle $T^\ast U \to U$. 


The strength of \cref{thm:GGinCOT} is related to the {\em parametric Morse lemma}, which says that up to a local change of coordinates any function germ $F \colon (\R^K,0) \to (\R,0)$ with critical point at $0$ can be split into a nondegenerate quadratic function $Q_F$ and a fully degenerate part $f$, i.e.\ a representative of the germ $F$ can be brought into the form $f(x^1,\ldots,x^{k(F)}) + Q_F(x^{k(F)+1},\ldots, x^K)$, where $K-k(F)$ is the rank of the Hessian matrix of $F$ at 0 \cite[\S 14.12]{brocker1975}. Therefore, $F,G \colon (\R^K,0) \to (\R,0)$ are right equivalent if and only if the signatures of $Q_F$ and $Q_G$ coincide and their fully reduced parts $f$ and $g$ fulfil the geometric condition in \cref{thm:GGinCOT}.

Thus, \cref{thm:GGinCOT} is very appealing because it allows us to turn an analysis problem into a geometric problem. 
The geometric problem itself, i.e.\ the description of intersecting Lagrangian manifolds, is, however, important in its own right. 
For instance, in dynamical systems, where intersections of Lagrangian invariant manifolds in phase spaces encode important information about the dynamics of the system \cite{Haro2000,Lomel2008}.
Global aspects such as lower bounds for the number of intersection points of Lagrangian intersections (e.g.\ the Arnold--Givental conjecture \cite{Frauenfelder2004}) have motivated many developments in Symplectic Topology such as Flour homology \cite{fukaya2010lagrangian}, the theory of pseudo-holomorphic curves, and global methods based on generating functions. We refer to \cite{LagrangianInterTheo} and references therein. In this paper, however, we study local aspects only.

Boundary value problems in Hamiltonian systems can be phrased as Lagrangian intersection problems and local properties of the intersections are of high significance for a description of the bifurcation behaviour of solutions \cite{WeinstBvPHam}.
It is, therefore, desirable, to obtain a reverse direction of \cref{thm:GGinCOT}, i.e.\ a statement which allows to turn the description of (possibly non-tangentially) intersecting Lagrangian manifolds in arbitrary symplectic manifolds into a problem in classical singularity theory or catastrophe theory.
This paper provides an important theoretical foundation for works in Numerical Analysis and Geometric Numerical Integration, where for the computation of bifurcation diagrams of solutions to Hamiltonian boundary value problems a correct exploitation of symplectic structure is crucial \cite{PDEBifur,bifurHampaper,obstructionPaper,numericalPaperShort,
	numericalPaper,PhDThesis}.

We will show that we can assign smooth function germs to local Lagrangian intersection problems such that the following theorem holds.

\begin{theorem}\label{thm:CorrSingularitiesDetail}
Let $X,\Lambda$ and $X',\Lambda'$ be Lagrangian submanifolds of a symplectic manifold $Z$ such that $\Lambda$ intersects $X$ in an isolated point $z$ and $\Lambda'$ intersects $X'$ in an isolated point $z'$. Let $f$ be the function germ assigned to the problem $z \in X \cap \Lambda$ and $f'$ be the function germ assigned to the problem $z' \in X' \cap \Lambda'$. The germs $f$ and $f'$ are stably right equivalent\footnote{\label{footnote:stablyrequiv}Two function germs are stably right equivalent if they become right equivalent after adding nondegenerate quadratic forms in new variables (see \cref{def:stablyrequiv}).} if and only if there exists a local symplectomorphism $\Phi$ mapping an open neighbourhood $U$ of $z$ to an open neighbourhood $U'$ of $z'$ with
\[
\Phi(z)=z', \quad \Phi(X\cap U) = X'\cap U', \quad \Phi(\Lambda \cap U)= \Lambda' \cap U'.
\]
\end{theorem}

In other words (notions will be made precise later):

\begin{restatable}{theorem}{CorrSingularities}
\label{thm:CorrSingularities}
There exists a 1-1 correspondence between Lagrangian contact problems modulo stably contact equivalence and smooth real-valued function germs up to stably right equivalence.
\end{restatable}

\begin{example}\label{ex:IntersectionExample}
Consider the symplectic manifold $Z=(\R^4, \Omega = \d x_1 \wedge \d \xi_1 + \d x_2 \wedge \d \xi_2)$. The Lagrangian submanifolds
\[
X = \{ (x_1,x_2,\xi_1,\xi_2) \, | \, x_1=0, x_2=0\}, \quad \Lambda = \{ (x_1,x_2,\xi_1,\xi_2) \, | \, \xi_1=x_1^2, \xi_2=x_2\}
\]
have an isolated intersection point $z=(0,0,0,0)$ (see \cref{fig:IntersectionExample}).
\begin{figure}
	\includegraphics[width=0.5\textwidth]{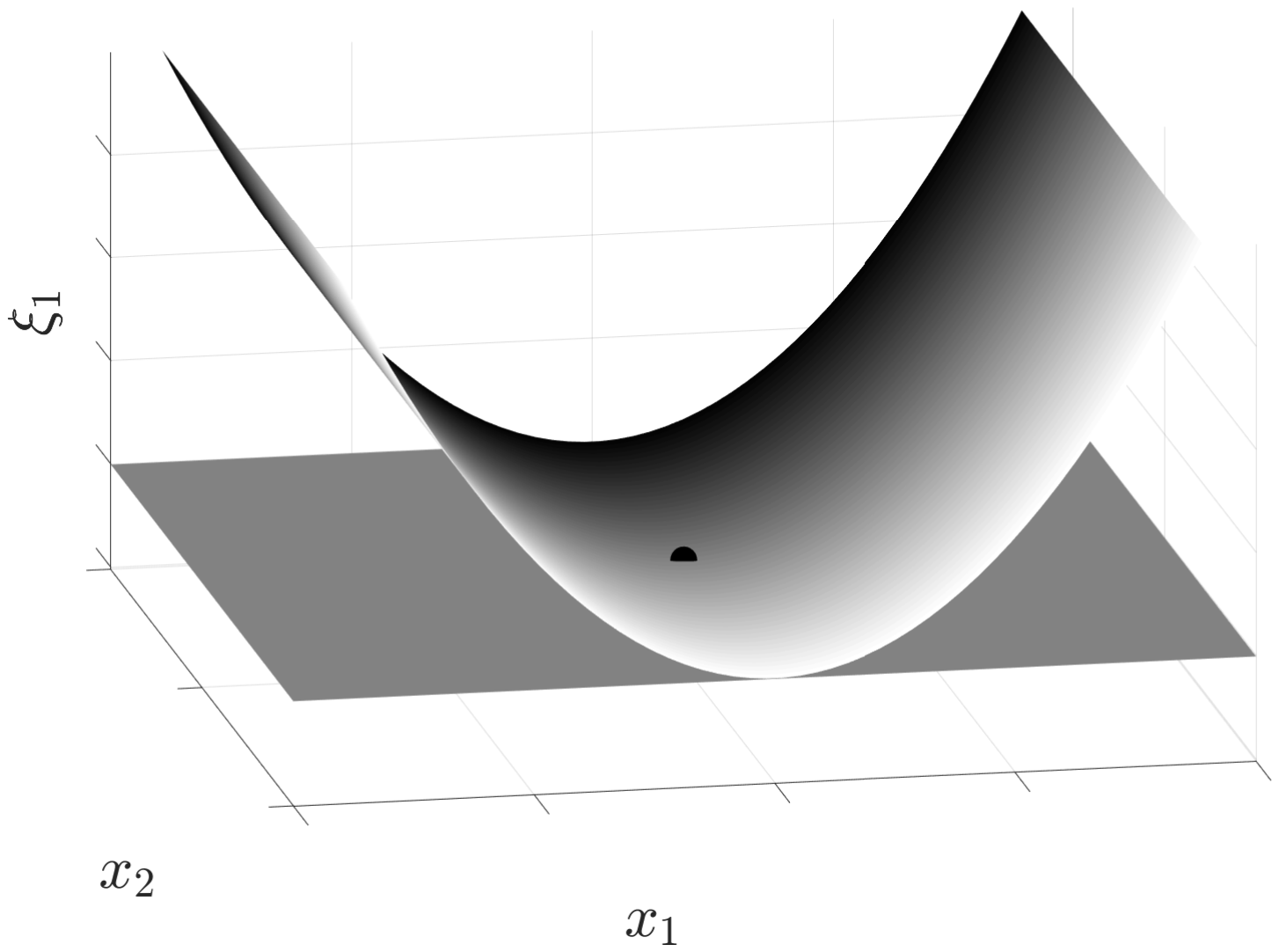}
	\caption{Illustration of the Lagrangian intersection considered in \cref{ex:IntersectionExample}. Values of the dimension $\xi_2$ relate to the shading of the displayed surfaces. The intersection point is isolated and denoted as a black dot.}\label{fig:IntersectionExample}
\end{figure}

We associate a function $f$ whose germ at zero characterises the type of intersection up to symplectomorphisms as follows:
let $T^\ast X$ denote the cotangent bundle over $X$ with the canonical symplectic structure and Darboux coordinates $q^1,q^2,p_1,p_2$. The map $\Psi \colon Z \to T^\ast X$ with coordinate expression $(q^1,q^2,p_1,p_2) = \Psi(x_1,x_2,\xi_1,\xi_2) = (x_1,x_2,\xi_1,\xi_2)$ is a symplectomorphism such that $\Psi(X)$ is the zero section of $T^\ast X$ and $\Psi(\Lambda) = \{ (q^1,q^2,p_1,p_2) \, | \, p_1=(q^1)^2, p_2=q_2\}$ is the graph of the section $\d f \colon X \to T^\ast X$ for $f \colon X \to \R$ with $f(q^1,q^2) = \frac{1}{3} (q^1)^3+\frac{1}{2} (q^2)^2$.

The germ at $(q^1,q^2) =(0,0)$ of the function $f$ characterises the type of intersection of $X$ and $\Lambda$ in $z$. However, $f$ depends on the identification of $Z$ with $T^\ast X$, i.e.\ on the choice of a cotangent bundle structure: consider the symplectomorphism $\Psi' \colon Z \to T^\ast X$ with coordinate expression $(q^1,q^2,p_1,p_2) = \Psi'(x_1,x_2,\xi_1,\xi_2) = (x_1,x_2-2 \xi_2,\xi_1,\xi_2)$. Again, $\Psi'(X)$ is the zero section of $T^\ast X$. However, $\Psi'(\Lambda) = \{ (q^1,q^2,p_1,p_2) \, | \, p_1=(q^1)^2, p_2=-q_2\}$ is the graph of the section $\d f' \colon X \to T^\ast X$ for $f' \colon X \to \R$ with $f'(q^1,q^2) = \frac{1}{3} (q^1)^3-\frac{1}{2} (q^2)^2$.

The germs of $f$ and $f'$ cannot be right equivalent since their quadratic parts have different signatures. However, they are stably right equivalent\footnote{See footnote \ref{footnote:stablyrequiv}.} and, as we will show, the stably right-equivalence class of the germ of $f$ characterises this type of intersection exactly up to symplectomorphisms.
\end{example}

The example illustrates the challenge of ill-defined quadratic parts of function germs which characterise the intersection. The problem is {\em not} visible in the case of tangential Lagrangian intersections in a symplectic manifold or when the ambient symplectic manifold is a cotangent bundle already such that no identification needs to be chosen.

\Cref{thm:CorrSingularitiesDetail,thm:CorrSingularities} overcome the following issues which occur when trying to reverse \cref{thm:GGinCOT}.

\begin{itemize}

\item
The Lagrangian manifold $\d f(U)$ from \cref{thm:GGinCOT} intersects the zero-section of $U \subset T^\ast U$ tangentially at 0. Golubitsky and Guillemin's proof of the {\em only if} direction fails when the intersection of $U$ with $\d f(U)$ is not tangential. Indeed, the equivalence relation {\em right equivalence} considered in \cite{golubitsky1975contact} needs to be relaxed to {\em stably right equivalence} when non-tangential intersections are allowed.

\item
Let $X,\Lambda$ be Lagrangian submanifolds of a symplectic manifold $Z$ such that $\Lambda$ intersects $X$ in an isolated point $z$. After shrinking all manifolds around $z$, if necessary, there are many ways of mapping $Z$ symplectically to a neighbourhood of the zero-section $X \subset T^\ast X$. We will refer to a particular choice of such an identification as a {\em choice of a cotangent bundle structure}.
\begin{itemize}
\item
In the tangential case, $\Lambda$ will be transversal to the fibres of any cotangent bundle structure. This is not true anymore in the non-tangential case but we will show that suitable cotangent bundle structures always exists such that $\Lambda$ is the image of a section $\d f$ in the bundle $T^\ast X \to X$.

\item
The function $f$ is not defined independently of the auxiliary cotangent bundle structure but we will show that its stably right equivalence class is well-defined.
\end{itemize}

\end{itemize}

The results may be compared with the correspondence of embeddings of a Lagrangian submanifold into a cotangent bundle with catastrophe theory \cite{WeinsteinBates} or other classifications such as {\em Lagrangian singularities} in \cite{Arnold2012}, where the bundle structure is fixed. There, singularities occur because Lagrangian submanifolds fail to be projectable and intersect non-transversally with fibres.
In contrast, in this paper cotangent bundle structures are just of an auxiliary nature. 

Furthermore, we provide a parameter-dependent version of the results and relate families of intersecting Lagrangian manifolds to unfoldings of singular function germs.

\begin{restatable}{theorem}{CorBifur}
\label{thm:CorBifur}
There exists a 1-1 correspondence between parameter-dependent Lagrangian contact problems up to stably right equivalence and unfoldings of smooth, real-valued function germs up to stably right equivalence as unfoldings.
\end{restatable}

\begin{figure}
	
	\includegraphics[width=0.3\textwidth]{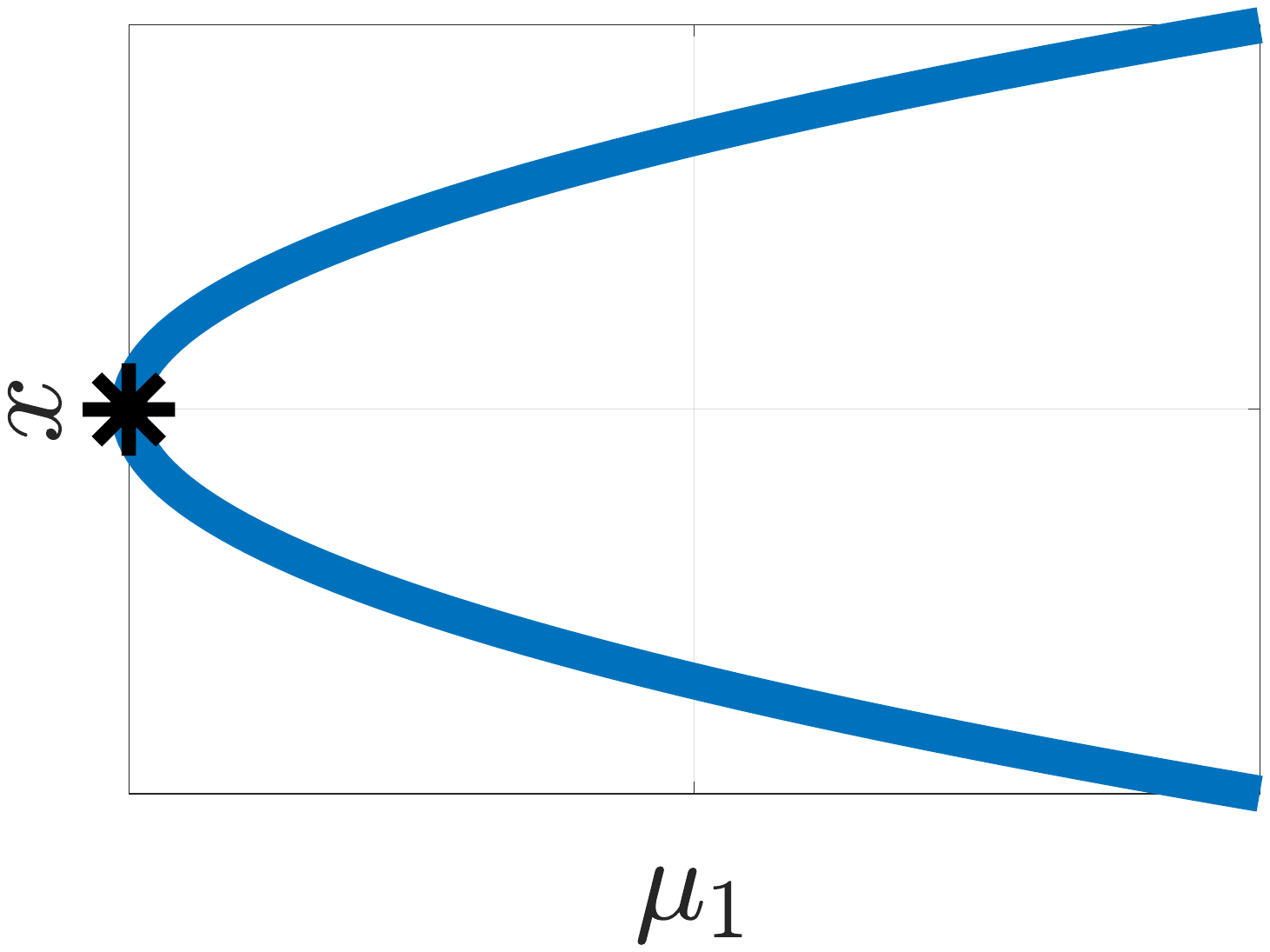}
	\includegraphics[width=0.3\textwidth]{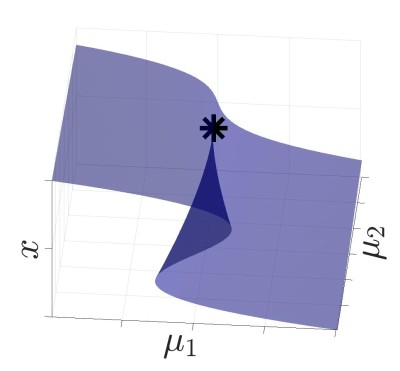}
	\includegraphics[width=0.3\textwidth]{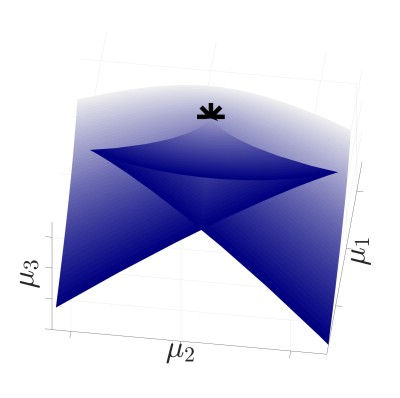}
	\includegraphics[width=0.3\textwidth]{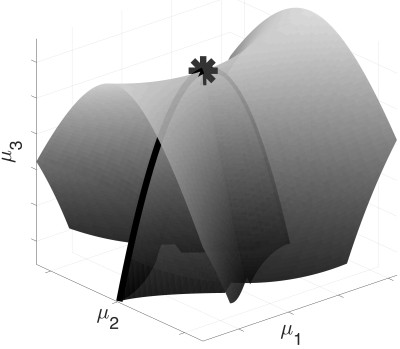}
	\includegraphics[width=0.3\textwidth]{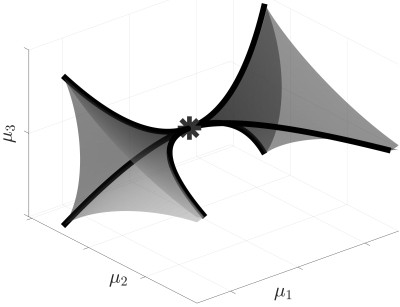}
	
	\caption{Illustrations of the normal forms of versal unfoldings of all singularities (catastrophes) whose versal unfoldings need no more than three parameters. From left to right these are fold ($A_2$), cusp ($A_3$), and swallowtail ($A_4$) in the first row and hyperbolic umbilic $(D_4^+)$ and elliptic umbilic $(D_4^-)$ in the second row. For $A_2$ and $A_3$ we plot the critical set of their normal forms $x \mapsto x^3 + \mu_1 x$ and $x \mapsto x^4 + \mu_2 x^2 + \mu_1 x$ over the parameter space. For the remaining catastrophes we plot a projection of the degenerate critical points of their normal forms $x \mapsto x^5 + \mu_3 x^3 + \mu_2 x^2 + \mu_1 x$ and $(x,y)\mapsto x^3 \pm xy^2 + \mu_3(x^2 \mp y^2)+\mu_2y+\mu_1x$ to the parameter space \cite{Arnold1992}. For animations see \cite{YoutubeSingularitiesAnimations}. }\label{fig:catastrophies}
	
\end{figure}

This allows transporting the highly-developed notions and algebraic framework of catastrophe theory to Lagrangian contact problems and bifurcations of Lagrangian intersection problems. As a corollary, classification results for singularities apply to Lagrangian contact problems. \Cref{fig:catastrophies} shows illustrations of the normal forms of versal unfoldings of all singularities whose versal unfoldings need at most three parameters. The type of intersection in \cref{ex:IntersectionExample} corresponds to the fold catastrophe $(A_2)$.



The remainder of the paper is structured as follows. In \cref{sec:LCPgerms} we first review some of Golubitsky and Guillemin's results and then prove that not necessarily tangential Lagrangian contact problems in arbitrary symplectic manifolds up to contact equivalence correspond to function germs up to stably right equivalence.
In \cref{sec:ParamLCP} we extend the identification results of \cref{sec:LCPgerms} to families of Lagrangian contact problems. In \cref{sec:Conclusion} we conclude the \cref{thm:CorrSingularities,thm:CorBifur} and show an application to boundary value problems for symplectic maps.


\section{Lagrangian contact problems and catastrophe theory}\label{sec:LCPgerms}

Let us introduce the notion of {\em Lagrangian contact problems} and review some definitions based on \cite{golubitsky1975contact,MatherIII}.

\begin{definition}[Lagrangian contact problem]\label{def:LagContProb}
Let $X$, $\Lambda$ be two Lagrangian submanifolds of a symplectic manifold $Z$ intersecting in an isolated point $z \in \Lambda \cap X$.
Then $(X,\Lambda,z)$ is called a {\em Lagrangian contact problem (in $Z$)}.
We say {\em $\Lambda$ has contact with $X$ in $z$}.
In the special case where $X$ and $\Lambda$ are tangential at $z$ the problem $(X,\Lambda,z)$ is called a {\em tangential Lagrangian contact problem}.
\end{definition}

\begin{definition}[contact equivalence of Lagrangian contact problems]\label{def:ContEQ}
Let $(X,\Lambda,z)$ and $(X',\Lambda',z')$ be two Lagrangian contact problems in $Z$ and $Z'$, respectively. We say that $(X,\Lambda,z)$ and $(X',\Lambda',z')$ are {\em contact equivalent} or $\Lambda$ {\em has the same contact} with $X$ at $z$ as $\Lambda'$ has contact with $X'$ at $z'$ if there exist open neighbourhoods $U \subset Z$ of $z$ and $U' \subset Z'$ of $z'$ and a symplectic diffeomorphism $\Phi\colon U \to U'$ such that
\[\Phi(z) = z', \quad \Phi(X\cap U) = X'\cap U', \quad \Phi(\Lambda \cap U)=\Lambda' \cap U'.\qedhere\]
\end{definition}

\begin{definition}[stably right equivalence of function germs]\label{def:stablyrequiv}
Two germs of smooth functions $f\colon (\R^k,0) \to (\R,0), \, x \mapsto f(x)$ and $g\colon (\R^l,0) \to (\R,0),\, y \mapsto g(y)$ are {\em stably right equivalent} if there exist nondegenerate quadratic forms $Q_F(u)$ and $Q_G(v)$ such that the germs at $0$ of $F(x,u) \coloneqq f(x)+Q_F(u)$ and $G(y,v) \coloneqq g(y)+Q_G(v)$ are right equivalent.
\end{definition}

{\em Stably right equivalence} constitutes an equivalence relation on the set of germs of smooth functions $f \colon (\R^k,0) \to (\R,0)$, $k \in \N_0$. Each stably right equivalence class $[\mathbf{g}]$ has a representative $\mathbf{g}$, which is the germ of a smooth function $g \colon (\R^{k_0},0) \to (\R,0)$ such that $k_0$ is minimal. The germ $\mathbf{g}$ and the function $g$ are called {\em fully reduced}. 

\begin{theorem}\label{thm:CanellationLaw}
Two germs of fully reduced smooth functions $f\colon (\R^k,0) \to (\R,0)$ and $g\colon (\R^k,0) \to (\R,0)$ are {\em stably right equivalent} if and only if the germs of $f$ and $g$ are right equivalent.
\end{theorem}

\begin{proof}
The statement follows from \cite{Weinstein1971}.
\end{proof}

\begin{remark}[Warning]
Contact equivalence for Lagrangian contact problems is not to be confused with Mather's notion of contact equivalence for function germs which is related to the contact of a smooth manifold with a zero section of a smooth bundle (i.e.\ without symplectic structure) \cite{MatherIII}. 
\end{remark}

\begin{definition}[cotangent bundle structure and projection]\label{def:cobundlestructure}
Let $X$ be a Lagrangian submanifold of a symplectic manifold $Z$. A symplectomorphism $\Phi$ from an open neighbourhood $U\subset Z$ with $X \subset U$ to a tubular neighbourhood of the zero section of $\pi \colon T^\ast X\to X$ is called a {\em cotangent bundle structure for $Z$ over $X$.} It is required that $\pi\circ \Phi \circ \iota$ is the identity map on $X$, where $\iota \colon X \hookrightarrow U$ is the natural inclusion.
The map $\pi \circ \Phi$ is called the {\em cotangent bundle projection with respect to the structure $\Phi$.}
\end{definition}

\begin{remark}
The existence of cotangent bundle structures (for a sufficiently small neighbourhoods $U$ of $X$) is guaranteed by Weinstein's generalisation of Darboux's theorem, see, for instance, \cite[Thm.\ 15.3]{Libermann1987}.
Notice that, in comparison to the existence of local Darboux coordinates, Weinstein's generalisation has a global character since $X$ does not need to be shrunk for cotangent bundle structures to exist. As all considerations in this paper are local, the globality is only exploited for notational convenience in this context. Moreover, we may refer to a function germ and a representative with the same symbol where a differentiation is not essential.
\end{remark}

\begin{remark}\label{rem:cobundleStructure1Form}
Cotangent bundle structures for a symplectic manifold $Z$ over a Lagrangian submanifold $X$ correspond to completely integrable Lagrangian subbundles of the tangent bundle $TZ$ transverse to $X$. (See \cite{Libermann1987} for details.) Moreover, cotangent bundle structures are in 1 to 1 correspondence with 1-forms defined on neighbourhoods of $X$ whose zero-set is $X$ and whose exterior derivative is the symplectic structure. 
The cotangent bundle structure $\Phi \colon U \to V$, where $U$ is an open neighbourhood of $X$ and $V$ is an open neighbourhood of the zero section of $T^\ast X$, associated to such a 1-form $\alpha$ is such that $\Phi^\ast \lambda = \alpha$ \cite{golubitsky1975contact,Libermann1987}, where $\lambda$ is the canonical/tautological 1-form of $T^\ast X$.
\end{remark}

For reference, let us recall some of Golubitsky and Guillemin's results and provide proofs for the cases where we would like to use the statement in a more general setting than in the original paper.

\begin{lemma}[{\cite[Lemma 3.1]{golubitsky1975contact}}]\label{lem:CompareCotStr}
Let $X$ be a Lagrangian submanifold of a symplectic manifold $Z$.
Let $x_1,\ldots,x_n$ be local coordinates on $X$ centred at $z$.
Consider two cotangent bundle structures $\Phi^\alpha$ and $\Phi^\beta$ on $Z$ over $X$ around $z$.
Let $\xi_1,\ldots,\xi_n$ be the conjugate momenta to $x_1,\ldots,x_n$ with respect to the first structure. Let $\lambda$ denote the canonical 1-form on $T^\ast X$, $\alpha \coloneqq (\Phi^\alpha)^\ast \lambda$, $\beta \coloneqq (\Phi^\beta)^\ast \lambda$.
The closed 1-form $\beta-\alpha$ can locally be written as $\d H$ with 
\begin{equation}\label{eq:ExprHinLemma}
H(x,\xi) = \sum_{i,j=1}^n h_{ij}(x,\xi) \xi_i \xi_j.
\end{equation}
If the Lagrangian submanifold $\Lambda \subset Z$ is the image of a section $\d \phi^\alpha$ with respect to the $\Phi^\alpha$-structure as well as the image of a section $\d \phi^\beta$ with respect to the $\Phi^\beta$-structure then
\begin{equation}\label{eq:RelPhiaPhib}
(\phi^\beta \circ k_{\alpha \beta})(x) = \phi^\alpha(x) + H(x,\nabla \phi^\alpha(x)) + \mathrm{const.},
\end{equation}
where $k_{\alpha\beta}$ is the diffeomorphism $k_{\alpha\beta} = \pi^\beta \circ (\pi^\alpha|_\Lambda)^{-1}$. Here $\pi^\alpha$ and $\pi^\beta$ are the projections corresponding to $\Phi^\alpha$ and $\Phi^\beta$, respectively.
\end{lemma}

\begin{proof}
The 1-forms $\alpha$ and $\beta$ are primitives of the symplectic form on $Z$. Moreover, $\alpha|_z =0$ if and only if $z \in X$. Analogously for $\beta$. Therefore, $\beta-\alpha$ is closed and has a local primitive $H$ on an open neighbourhood of $z$. The primitive must be of the form \eqref{eq:ExprHinLemma}: in local coordinates we have
\[
\alpha = \sum_j \xi_j \d x_j, \quad \beta = \sum_j (b_j \d x_j + c_j \d \xi_j), \quad \d H = \sum_j \left(\frac{\p H}{\p x_j} \d x_j + \frac{\p H}{\p \xi_j} \d \xi_j\right).
\]
Now
\begin{equation}\label{eq:alpbetadH}
\sum_j ((\xi_j-b_j) \d x_j + c_j \d \xi_j) = \alpha-\beta = -\d H 
= -\sum_j \left(\frac{\p H}{\p x_j} \d x_j + \frac{\p H}{\p \xi_j} \d \xi_j\right).
\end{equation}
Setting $\xi = 0$ in \eqref{eq:alpbetadH} and using $c_j(x,0) \equiv 0$ (which follows from $\beta|_{(x,\xi)=(x,0)}=0$) we deduce that $\frac{\p H}{\p \xi_j}(x,0) \equiv 0$. 
By Taylor's theorem there exists a primitive of the form \eqref{eq:ExprHinLemma}.

Let $\iota \colon \Lambda \hookrightarrow Z$ denote the embedding of $\Lambda$ into $Z$. The 1-forms $\iota^\ast \alpha$ and $\iota^\ast \beta$ are closed since $\Lambda$ is Lagrangian. Locally around a point of interest $\iota^\ast \alpha$ and $\iota^\ast \beta$ have primitives which we denote by $\varphi^\alpha$ and $\varphi^\beta$, respectively. 
Due to $\iota^\ast \beta-\iota^\ast \alpha=\d (H\circ \iota)$ we have
\begin{equation}\label{eq:RelonLambda}
\varphi^\beta = \varphi^\alpha + H \circ \iota+ \mathrm{const.}
\end{equation}
on $\Lambda$. Moreover, $\phi^\alpha = \varphi^\alpha \circ (\pi^\alpha|_\Lambda)^{-1}+ \mathrm{const}$ and $\phi^\beta = \varphi^\beta \circ (\pi^\beta|_\Lambda)^{-1}+ \mathrm{const}$. 
Expressing relation \eqref{eq:RelonLambda} in the canonical coordinates $(x,\xi)$ of the $\alpha$-cotangent bundle structure yields \eqref{eq:RelPhiaPhib}.
\end{proof}

\begin{remark}
If the manifolds $X$ and $\Lambda$ intersect nontrivially and $x_1,\ldots,x_n$ are centred coordinates at an intersection point of $\Lambda$ and $X$ then the constant in \eqref{eq:RelPhiaPhib} vanishes if we choose $\phi^\alpha(0)=\phi^\beta(0)$. However, in view of the parameter-dependent setting in \cref{sec:ParamLCP} it is convenient {\em not} to restrict to this case for the formulation of \cref{lem:CompareCotStr}.
\end{remark}

The following lemma corresponds to \cite[Prop.4.2]{golubitsky1975contact} which is a result by Tougeron \cite[p.209]{Tougeron1968} and is also contained in \cite[Appendix]{Weinstein1972}.

\begin{lemma}\label{lem:HFormula}
Let
\[
H(x,\xi) = \sum_{i,j =1}^n h_{ij}(x,\xi) \xi_i \xi_j
\]
be defined on an open neighbourhood of the origin in $\R^n \times \R^n$. Consider a real valued map $\phi$ defined on an open neighbourhood of 0 in $\R^n$ such that $\phi(0)=0$, $\nabla \phi(0)=0$, $\mathrm{Hess}\, \phi (0) = 0$.
The map
\begin{equation}\label{eq:cotForumla}
\psi (x) = \phi(x)+H(x,\nabla \phi(x))
\end{equation}
is right equivalent to $\phi$ on a neighbourhood of the origin in $\R^n$ and the right equivalence fixes the origin.
\end{lemma}

\begin{proof}
To simplify notation, we set
\[
\overline H(x)= H(x,\nabla \phi(x)), \quad \overline h_{ij}(x) = h_{ij}(x,\nabla \phi(x)).
\]
We prove the assertion using the homotopy method: define
\begin{equation}\label{eq:DefPsi}
\psi_t(x) = \phi(x)+t \overline H(x).
\end{equation}
We seek a family of local diffeomorphisms $f_t$ fixing 0 such that
\begin{equation}\label{eq:HomotopyLemma}
\psi_t \circ f_t = \phi.
\end{equation}

Differentiating \eqref{eq:HomotopyLemma} with respect to $t$ we find
\[
\frac{\d}{\d t} (\psi_t) \circ f_t + \left\langle \nabla \psi_t \circ f_t , \frac{\d}{\d t} f_t \right\rangle =0.
\]
Here, $\langle \cdot,\cdot \rangle$ denotes the scalar product in $\R^n$. An evaluation at $f_t^{-1}(x)$ yields
\begin{equation}\label{eq:SolveforW}
\frac{\d }{\d t} \psi_t  + \langle \nabla \psi_t , w(x,t) \rangle =0
\end{equation}
with
\begin{equation}\label{eq:Solveforft}
w(x,t)= \frac{\d}{\d t} (f_t)(f_t^{-1}(x)).
\end{equation}

The relation \eqref{eq:SolveforW} is an implicit expression for $w$ in terms of the given function $\psi_t$ and its derivatives. We will show that \eqref{eq:SolveforW} can be solved for $w$ such that $w(x,t)$ is defined around $x=0$ and $w(0,t)=0$. Then there exists an open neighbourhood $U\subset \R^n$ of $0$ such that the initial value problem
\begin{equation}
\frac{\d}{\d t} f_t (x) = w(f_t(x),t), \quad f_0(x)=x
\end{equation}
can be solved for all $x\in U$ on the interval $t \in [0,1]$. The obtained family of functions $f_t$ fulfils $\frac{\d}{\d t} (\psi_t \circ f_t) = 0$ with $f_0=\id$ and, therefore, \eqref{eq:HomotopyLemma}. Moreover, $f_t(0)\equiv 0$ so that $f_1$ is the required right equivalence.

We now show that for $x$ near $0$ expression \eqref{eq:SolveforW} is solvable for $w$ such that $w(0,t)=0$. Differentiating \eqref{eq:DefPsi} with respect to $x$ yields
\begin{equation}\label{eq:relateNablaPsis}
\begin{split}
\frac{\p \psi_t}{\p x_l}
&= \frac{\p \phi}{\p x_l}
  + t \sum_{i,j} \left(\frac{\p \overline h_{ij}}{\p x_l}  \frac{\p \phi}{\p x_i}\frac{\p \phi}{\p x_j} + 2 \overline h_{ij} \frac{\p^2 \phi}{\p x_i \p x_l} \frac{\p \phi}{\p x_j}\right)\\
&= \sum_j \underbrace{\left( \delta_{lj}+ t \sum_{i} \left(\frac{\p \overline h_{ij}}{\p x_l}  \frac{\p \phi}{\p x_i} + 2 \overline h_{ij} \frac{\p^2 \phi}{\p x_i \p x_l}\right)\right)}_{=:B_{lj}(t,x)} \frac{\p \phi}{\p x_j}
\end{split}
\end{equation}
The maps $B_{lj}(t,x)$ form a matrix $B$ with $B(t,0)=\Id$. Therefore, there exists a neighbourhood of $x=0$ such that $B$ is invertible for all $t\in[0,1]$. We have
\begin{equation}\label{eq:RelateDpsiDpsit}
\nabla \phi = B^{-1} \nabla \psi_t.
\end{equation}
The functions $\overline h_{ij}$ constitute a matrix which we denote by $\mathcal H$. 
Differentiating \eqref{eq:DefPsi} with respect to $t$ and using \eqref{eq:RelateDpsiDpsit} we get 
\begin{equation*}
\frac{\d}{\d t} \psi_t 
= \nabla \phi^{\top} \mathcal H \nabla \phi
\stackrel{\eqref{eq:RelateDpsiDpsit}} = \nabla \psi_t^{\top} \underbrace{B^{-\top} \mathcal H \nabla \phi}_{=:-w},
\end{equation*}
where $B^{-\top}=(B^{-1})^\top$.
Now $w(0,t)=0$ and $w$ solves \eqref{eq:SolveforW}. This completes the proof.
\end{proof}

The \cref{lem:CompareCotStr,lem:HFormula} imply the following proposition.

\begin{proposition}\label{prop:TouchWellDefined}
Let $(X,\Lambda,z)$ be a tangential Lagrangian contact problem in $Z$. Consider two cotangent bundle structures over $X$ near $z$ such that $\Lambda$ is the image of the section $\d \phi^\alpha$ and $\d \phi^\beta$ with $\phi^\alpha(z)=0=\phi^\beta(z)$. Then $\phi^\alpha$ and $\phi^\beta$ are right equivalent.
\end{proposition}

\begin{remark}\label{rem:anyBundleSection}
As $X$ and $\Lambda$ are tangential at $z$, for any cotangent bundle structure the submanifold $\Lambda$ is the image of a section of the cotangent bundle locally around the point of contact. This is not true for general Lagrangian contact problems.
\end{remark}

\begin{theorem}\label{prop:TouchContactifRequiv}
Let $(X,\Lambda,z)$ and $(X,\Lambda',z)$ be two tangential Lagrangian contact problems in $Z$.
For any cotangent bundle structure over $X$ near $z$ such that $\Lambda$ is the image of the section $\d \phi$ and $\Lambda'$ the image of $\d \phi'$ with $\phi(z)=0=\phi'(z)$,
the function germs $\phi$ and $\phi'$ are right equivalent if and only if the tangential Lagrangian intersection problems are contact equivalent.
\end{theorem}

\begin{proof} In the following, we may shrink $Z,X,\Lambda,\Lambda'$ repeatedly to intersections with open neighbourhoods of $z$ without mentioning.
Assume $\phi = \phi' \circ r$ for a right equivalence $r$ such that $r$ fixes $z$. Its cotangent lifted map\footnote{The cotangent lift $\Psi\colon T^\ast X' \to T^\ast X$ of a diffeomorphism $\psi \colon X \to X'$ is defined as $\Psi(\alpha') = \alpha' \circ \d \psi|_{\psi^{-1}(\pi'(\alpha'))}$, where $\pi'\colon T^\ast X' \to X'$ is the bundle projection (see \cite[\S 6.3]{Marsden1999}).\label{ft:CotLift}} $R$ fixes $X$ and
\[
R(\d \phi'|_x) 
= \d \phi'|_x \circ \d r|_{r^{-1}(x)} 
= \d (\phi' \circ r)|_{r^{-1}(x)} 
= \d \phi|_{r^{-1}(x)}, \qquad x \in X.
\]
Therefore, the symplectic diffeomorphism $R$ maps $\Lambda'$ to $\Lambda$ and, thus, provides a contact equivalence between $(X,\Lambda,z)$ and $(X,\Lambda',z)$.

Now assume there exists a symplectic diffeomorphism with $\Phi(X)=X$, $\Phi(z)=z$ and $\Phi(\Lambda)=\Lambda'$. Choose a cotangent bundle structure with projection denoted by $\pi \colon Z \to X$ such that $\Lambda$ is the image of the section $\d \phi$ and $\Lambda'$ the image of the section $\d \phi'$ around $z$ with $\phi(z)=0=\phi'(z)$.
Using $\Phi$ we can obtain a new cotangent bundle structure with bundle projection $\pi' \colon Z \to X$ such that $\pi' = \pi \circ \Phi^{-1}$. 
The map $\Phi \circ \d \phi$ maps $X$ onto $\Lambda'$ and is a section of $\pi' \colon Z \to X$. This means $\Lambda'$ can be represented by $\d \phi$ in the new structure. The situation is illustrated in \cref{fig:CotStr}.
\begin{figure}
\begin{center}
\[
\begin{tikzcd}
Z\arrow{rr}{\Phi}
\arrow{dr}{\Psi}
\arrow[bend right=45,swap]{ddr}{\pi}
&
&Z\arrow[swap]{dl}{\Psi' }
\arrow[bend left=45]{ddl}{\pi'}\\
&T^\ast X\arrow{d} &\\
& X &
\end{tikzcd}
\]
\caption{Illustration to the second part of the proof of \cref{prop:TouchContactifRequiv}.
$Z$ has two cotangent bundle structures given by $\Psi$ and $\Psi' = \Psi \circ \Phi^{-1}$. The image in $T^\ast X$ of the section $\d \phi$ is identified with $\Lambda \subset Z$ if the bundle structure $\Psi$ is used and is identified with $\Lambda' \subset Z$ if the structure $\Psi'$ is used.}\label{fig:CotStr}
\end{center}
\end{figure}
Therefore, by \cref{prop:TouchWellDefined}, the function germs $\phi$ and $\phi'$ must be right equivalent.
\end{proof}

\begin{remark}
The hypothesis that $\Lambda$ and $\Lambda'$ are tangential to $X$ is used in two ways in the proof of \cref{prop:TouchContactifRequiv}: while it is convenient that in any contangent bundle structure $\Lambda$ and $\Lambda'$ are locally around their point of contact images of sections (see \cref{rem:anyBundleSection}), more importantly the hypothesis is needed to apply \cref{prop:TouchWellDefined} which, as it relies on \cref{lem:HFormula}, needs that the 2-jets of $\phi$ and $\phi'$ vanish at the point of contact.
\end{remark}


We recall the well-known 
{\em parametric Morse Lemma} or {\em Splitting Lemma} and establish some notation.

\begin{lemma}[parametric Morse lemma]\label{lem:ParametricMorse}
Let $\phi \colon (\R^n,0)\to(\R,0)$ be a function germ with critical point at the origin 0. Consider the decomposition $\R^n = \overline X \oplus \underline X$ for two linear subspaces $\overline X$ and $\underline X$ such that the Hessian matrix $B$ of the restriction $\phi|_{\underline X}\colon (\underline X,0)\to(\R,0)$ is invertible. There exists a change of coordinates $K$ on a neighbourhood of $0 \in \R^n$ of the form $K(\overline x,\underline x)=(\overline x,\kappa(\overline x,\underline x))$ with $K(0, 0)=( 0,0)$ such that
\[
(\phi \circ K)(\overline x,\underline x) = f(\overline x )+{\underline x}^{\top} B \underline x.
\]
If we choose the dimension of $\underline X$ to be maximal with respect to the property that $\phi|_{\underline{X}}$ is invertible\footnote{We can choose $\dim \underline X$ to be the rank of the Hessian matrix of $\phi$ at $0$} then the 2-jet of $f$ vanishes. 
\end{lemma}

A proof is given in \cite[\S 14.12]{brocker1975} or in \cite[Lemma 5.12]{Wassermann1974}. A fibred version including uniqueness results can be found in \cite{Weinstein1971}.

We leave the setting of tangential Lagrangian contact problems and extend \cref{prop:TouchWellDefined}: a function germ assigned to a (not necessarily tangential) Lagrangian intersection problem using any cotangent bundle structure for which the intersection problem is graphical is well-defined up to stably right equivalence. For this we first prove the following lemma.

\begin{lemma}\label{lem:SwitchOn}
On $\R^n = \overline X \oplus \underline X$ consider
\begin{itemize}
\item
coordinates $x=(\overline x, \underline x)=((x^1,\ldots,x^k),(x^{k+1},\ldots,x^n))$,
\item
a nondegenerate symmetric matrix $Q \in \R^{(n-k) \times (n-k)}$,
\item
a function $g \colon \R^k \to \R$ whose 2-jet vanishes at 0,
\item
and a matrix valued function $\mathcal H \colon \R^n \to \mathrm{Sym}(n)$ with 
\[
\mathcal H(x) = (h_{ij}(x))_{i,j=1,\ldots,n}
=
\begin{pmatrix}
H^{11}(x) & H^{12}(x) \\
{H^{12}}(x)^{\top}&0
\end{pmatrix} \in \mathrm{Sym}(n)  \subset \R^{n \times n}
\]
for $H^{11}(x) \in \mathrm{Sym}(k)$ and $H^{12}(x) \in \R^{k \times (n-k)}$.
\end{itemize}
For $t \in \R$ let
\begin{align*}
\psi (x)    & = g(\overline x) + \underline x^{\top} Q \underline x \\
\psi_t (x) & = g(\overline x) + \underline x^{\top} Q \underline x + t ( \nabla \psi(x)^{\top} \mathcal H(x) \nabla \psi(x)).
\end{align*}
Then $\psi_t$ is right equivalent to $\psi=\psi_0$ around $x=0$ and the right equivalence fixes $x=0$.
\end{lemma}

\begin{proof}
Motivated by the proof of \cref{lem:HFormula} we define the components $B_{lj}(t,x)$ of a matrix $B(t,x) \in \R^{n\times n}$ as
\[
B_{lj} = \delta_{lj}+ t \sum_{i} \left(\frac{\p  h_{ij}}{\p x_l}  \frac{\p \psi}{\p x_i} + 2  h_{ij} \frac{\p^2 \psi}{\p x_i \p x_l}\right).
\]
We have
\begin{equation}\label{eq:BlocallyInvertible}
B(t,0)=\Id_n + 2 t \Hess \psi(0) \mathcal H(0)
= \begin{pmatrix}
\Id_k&0\\
4 t Q {H^{12}}(0)^{\top}& \Id_{n-k},
\end{pmatrix}
\end{equation}
which is invertible for all $t$. Using the same calculation as in \eqref{eq:relateNablaPsis} we obtain
\[
\nabla \psi_t (x) = B(t,x) \nabla \psi(x).
\]
Therefore,
\begin{equation}\label{eq:ddtpsit}
\frac{\d}{\d t} \psi_t 
= \nabla \psi ^{\top} \mathcal H \nabla \psi
= \nabla \psi_t^{\top} \underbrace{B^{-T} \mathcal H \nabla \psi}_{=: - \omega}
= -\langle \nabla \psi_t , \omega \rangle.
\end{equation}
There exists a neighbourhood $U$ of $0 \in \R^n$ such that the initial value problem
\[
\frac{\d }{\d t} f_t(x) = \omega(f_t(x),t), \quad f_0(x)=x
\]
is solvable for all $x \in U$ and all $t \in [0,1]$.
Since $\omega(0,t)=0$ we have $f_t(0)=0$ and
\begin{align*}
\frac{\d}{\d t} (\psi_t \circ f_t)
&\;=  \frac{\d}{\d t} (\psi_t) \circ f_t + \left\langle \nabla\psi_t \circ f_t , \frac{\d}{\d t} f_t \right\rangle\\
&\;=  \frac{\d}{\d t} (\psi_t) \circ f_t + \left\langle \nabla\psi_t  , \omega  \right\rangle \circ f_t\\
&\stackrel{\eqref{eq:ddtpsit}}=0.
\end{align*}
Since $f_0=\id_U$ we have
\[
\psi_t \circ f_t = \psi_0 \circ f_0 = \psi
\]
and $f_t$ is the required right equivalence.
\end{proof}


\begin{proposition}\label{prop:IndofCotStructure}
Let $(X,\Lambda,z)$ be a Lagrangian contact problem in $Z$.
After shrinking $Z$ to a sufficiently small open neighbourhood of $z$, consider two cotangent bundle structures over $X$ with projections $\pi^\alpha \colon Z \to X$, $\pi^\beta \colon Z \to X$ such that $\Lambda$ is given as the image of the section $\d \phi^\alpha$ and $\d \phi^\beta$, respectively, whereas $\phi^\alpha(z)=0=\phi^\beta(z)$.
Then the germs of $\phi^\alpha$ and $\phi^\beta$ at $z$ are stably right equivalent.
\end{proposition}

\begin{proof}
By the parametric Morse Lemma (\cref{lem:ParametricMorse}) there exist coordinates $x=(\overline x, \underline x)=((x^1,\ldots,x^k),(x^{k+1},\ldots,x^n))$ on $X$ centred at $z$ such that
\[
\phi^\alpha (x) = f(\overline x)+\underline x^{\top} B \underline x
\]
for a smooth function germ $f$ with vanishing 2-jet at $\overline x=0$ and an invertible symmetric matrix $B$. The function $\phi^\alpha$ is stably right equivalent to $f$.
By \cref{lem:CompareCotStr} we have
\begin{equation}\label{eq:ApplyCompCot}
\phi^\beta \circ k_{\alpha \beta} = \phi^\alpha + H(x,\nabla \phi^\alpha) ,
\end{equation}

for a function
\[
H(x,\xi) = \sum_{i,j =1}^n h_{ij}(x,\xi) \xi_i \xi_j
= \xi^{\top}\begin{pmatrix}
H^{11}(x,\xi)&H^{12}(x,\xi)\\
H^{12}(x,\xi)^{\top}&H^{22}(x,\xi)
\end{pmatrix}\xi
\]
with matrices $H^{11}(x,\xi)\in \mathrm{Sym}(k)$, $H^{22}(x,\xi)\in\mathrm{Sym}(n-k)$, and $H^{12}(x,\xi) \in \R^{k \times (n-k)}$ and with $k_{\alpha\beta} = \pi^\beta \circ (\pi^\alpha|_\Lambda)^{-1}$.
For $i,j \in \{1,2\}$ define
\[
\mathcal H(x) 
\coloneqq\mathcal H(\overline x, \underline x) 
\coloneqq H(x,\nabla \phi^\alpha(x)),
\quad
\mathcal H^{ij}(x) 
\coloneqq \mathcal H^{ij}(\overline x, \underline x) 
\coloneqq H^{ij}(x,\nabla \phi^\alpha(x)).
\]
Therefore,
\[
\mathcal H(x)  
=  \nabla_{\overline x}f(\overline x)^{\top} \mathcal H^{11}(x) \nabla_{\overline x} f(\overline x)
+ 4 \underline x^{\top} B \mathcal H^{12}(x)^{\top} \nabla_{\overline x} f(\overline x)
+ 4 \underline x^{\top} B \mathcal H^{22}(x)B \underline x.
\]

We calculate
\begin{equation}\label{eq:phibetacoords}
\begin{split}
\phi^\beta(k_{\alpha \beta}(\overline x, \underline x))
&\stackrel{\eqref{eq:ApplyCompCot}}= \phi^\alpha(\overline x, \underline x) + \mathcal H(x)\\
&\;\;\;=  f(\overline x)+\underline x^{\top} (B+4B\mathcal H^{22}(x)B) \underline x\\
&\;\;\;\phantom{=}+ \nabla_{\overline x}f(\overline x)^{\top} \mathcal H^{11}(x) \nabla_{\overline x} f(\overline x)
+ 4 \underline x^{\top} B \mathcal H^{12}(x)^{\top} \nabla_{\overline x} f(\overline x). 
\end{split}
\end{equation}

Let us define
\[
B'(x) = B+4B\mathcal H^{22}(x)B.
\]

The kernel of $\Hess (\phi^\alpha)$ and the kernel of $\Hess(\phi^\beta)$ at $(\overline x, \underline x)=(0,0)$ both describe the intersection $T_z X \cap T_z \Lambda$ (but in different coordinates). Indeed, the kernel of $\Hess(\phi^\beta \circ k_{\alpha \beta})$ must coincide with the kernel of $\Hess(\phi^\alpha)$ which is $\overline X = \{\underline x=0\}$. We calculate the Hessian matrix of $\phi^\beta \circ k_{\alpha \beta}$ at $(\overline x, \underline x)=(0,0)$ using \eqref{eq:phibetacoords} and obtain
\[
\Hess(\phi^\beta \circ k_{\alpha \beta})(0,0)
=\begin{pmatrix}
0&0\\0& 2B'(0)
\end{pmatrix}.
\]
Since $\Lambda$ is graphical in both cotangent bundle structures, i.e.\ a restriction of the bundle projections to $\Lambda$ yields a diffeomorphism between $\Lambda$ and $\pi(\Lambda)$ around $z$, $B'(0)$ must be invertible by a dimension argument. The function $\phi^\alpha$ is stably right equivalent to
\[
\phi^\alpha_1(x) \coloneqq f(\overline x) + \underline x^{\top} B'(0) \underline x.
\]
For $x$ in a sufficiently small neighbourhood of $0$ the signature of $B'(x)$ is constant. By Sylvester's law of inertia, there exists a smooth family of invertible matrices $A(x)$ such that
\begin{equation}\label{eq:SylvesterBprime}
A(x)^{-T}B'(x) A^{-1}(x) = B'(0)
\end{equation}
for all $x$ near 0. Consider
\[
r(\overline x,\underline x)=(\overline x, A(x) \underline x).
\]
The map $r$ fixes $x=0$ and is a diffeomorphism on a neighbourhood of $x=0$: the Jacobian matrix of $r$ is given by the block matrix
\[ \D r(x)=
\begin{pmatrix}
\Id_k & 0\\
(\p A(x))_{l=1,\ldots,k} & (\p A(x))_{l=k+1,\ldots,n} +A(x)
\end{pmatrix},
\]
where $(\p A(x))_{l=1,\ldots,k}$ denotes the first $k$ columns and $(\p A(x))_{l=k+1,\ldots,n}$ the remaining $n-k$ columns of an $(n-k)\times n$ matrix $\p A(x)$ whose $l$-th column is given as
\[
(\p A(x))_{l} = 
\frac{\p A}{\p x^l}(x) \underline x, 
\]
where the derivative $\frac{\p A}{\p x^l}$ is taken component-wise. Now the determinant of $\D r(0)$ coincides with the determinant of $A(0)$ which is non-zero, so $r$ is indeed a right equivalence.

Let us define
\begin{equation}\label{eq:deftildeH}
\begin{split}
\tilde {\mathcal H}^{11}(x) &\coloneqq \mathcal H^{11}(r^{-1}(x))\\
\tilde {\mathcal H}^{12}(x) &\coloneqq \mathcal H^{12}(r^{-1}(x)) B A(r^{-1}(x))^{-1} B'(0)^{-1}
\end{split}
\end{equation}
By \cref{lem:SwitchOn} the function $\phi^\alpha_1$ is right equivalent to
\begin{align*}
\phi^{\alpha}_2(x) &\coloneqq  f(\overline x)+\underline x^{\top} B'(0) \underline x\\
&\phantom{\coloneqq}+ \nabla_{\overline x}f(\overline x)^{\top} \tilde {\mathcal H}^{11}(x) \nabla_{\overline x} f(\overline x)
+ 4 \underline x^{\top} B'(0) \tilde {\mathcal H}^{12}(x)^{\top} \nabla_{\overline x} f(\overline x). 
\end{align*}

We have
\begin{equation}\label{eq:phi2alpha}
\begin{split}
(\phi^{\alpha}_2\circ r)(x)
&=  f(\overline x)+\underline x^{\top} A(x)^{\top} B'(0) A(x)\underline x
+ \nabla_{\overline x}f(\overline x)^{\top} \tilde {\mathcal H}^{11}(r(x)) \nabla_{\overline x} f(\overline x)\\
&\phantom{=}+ 4 \underline x^{\top} A(x)^{\top} B'(0) \tilde {\mathcal H}^{12}(r(x))^{\top} \nabla_{\overline x} f(\overline x).
\end{split}
\end{equation}
Comparing \eqref{eq:phi2alpha} with \eqref{eq:phibetacoords} shows that the function $\phi^{\alpha}_2\circ r$ coincides with $\phi^\beta \circ k_{\alpha \beta}$.
Thus, the functions $\phi^\alpha$ and $\phi^\beta$ are stably right equivalent.
\end{proof}

\begin{remark}
We see from the proof of \cref{prop:IndofCotStructure} that if the intersection of $\Lambda$ and $X$ is not tangential then the dimension of $\underline X$ is greater than 0 and there exist two cotangent bundle structures such that 
$\phi^\alpha$ and $\phi^\beta$ are stably right equivalent but {\em not} right equivalent. \Cref{ex:IntersectionExample} is an instance of this phenomenon. In the general case, to a cotangent bundle structure over $X$ defined\footnote{Recall from \cref{rem:cobundleStructure1Form} that suitable 1-forms define cotangent bundle structures.} by a canonical 1-form $\alpha$ for which $\Lambda$ is graphical choose another cotangent bundle structure $\beta = \alpha + \d H$ with
\[
H(x,\xi) 
= \xi^{\top}\begin{pmatrix}
0&0\\
0& \frac 14 (B^{-1}DB^{-1} - B^{-1})
\end{pmatrix}\xi,
\]
where $D$ is an invertible symmetric matrix which has a different signature than $B$. As in the proof of \cref{prop:IndofCotStructure}, the coordinates $(x,\xi)$ refer to Darboux coordinates with respect to the $\alpha$-structure. 
We get $B'=D$ which is invertible such that $\Lambda$ is graphical for the cotangent bundle structure defined by $\beta$. However, the signatures of $\Hess \phi^\alpha(0)$ and $\Hess \phi^\beta(0)$ do not coincide.
\end{remark}


For a Lagrangian contact problem $(X,\Lambda,z)$ in a symplectic manifold $Z$. There is always a contangent bundle structure $\pi \colon Z \to X$ such that $\Lambda$ is graphical, i.e.\ the image of a smooth section of $\pi \colon Z \to X$ locally around $z$.

\begin{lemma}[Existence of generic cotangent bundle structures]\label{lem:ExistanceCotStructure}
Let $Z$ be a symplectic manifold with a Lagrangian submanifold $X$, an isotropic submanifold $\Lambda$ and a point $z \in X \cap \Lambda$. 
Shrinking $Z$ around $z$, if necessary, there exists a projection $\pi \colon Z \to  X$ with Lagrangian fibres transversal to $X$ such that the restriction $\pi|_{\Lambda} \colon \Lambda \to \pi(\Lambda)$ is a diffeomorphism. In other words $\Lambda$ is graphical for $\pi$.
\end{lemma}

\begin{proof}
The proof is similar to \cite[Lemma 2.1]{bifurHampaper}.
Let $n = \dim  X$,
$k=\dim \Lambda$,
$n_1 = \dim(T_z X \cap T_z \Lambda)$
and
$n_2 =k-n_1$.
Let $q^1,\ldots,q^n$ be coordinates on $ X$ centred at $z$. Extend these to Darboux coordinates $q^1,\ldots,q^n,p_1,\ldots,p_n$ on $Z$ centred at $z$. 
This induces coordinates
$\left.\frac{\p }{\p q^1}\right|_z,
\ldots,
\left.\frac{\p }{\p q^n}\right|_z,
\left.\frac{\p }{\p p_1}\right|_z,
\ldots,
\left.\frac{\p }{\p p_n}\right|_z$
on the tangent space $T_z Z$ of which $T_z  X$ is a linear subspace. The manifold $ X$ can be identified with a neighbourhood of 0 of $T_z  X$.
By the classification of pairs of coisotropic\footnote{The normal forms used in this context are derived by passing to the symplectic complements in the normal forms given in \cite[p.\ 16]{lor2014coisotropic}.\label{note:RefCoIsoClass}} linear subspaces \cite{lor2014coisotropic,coisoPair}, there exists a linear, symplectic change of coordinates $L$ of the coordinates on $T_z Z$ such that $T_z  X$ and $T_z  \Lambda$ are represented by the column-span of the following two matrices
\[
\begin{pmatrix}
\phantom{AA} \\[-0.75em]
\phantom{AA} & I_n & \phantom{AA} \\
\phantom{AA} \\[-0.75em]
& 0_{n\times n}\\
\phantom{AA} \\[-0.75em]
\end{pmatrix}
,
\quad
\begin{pmatrix}
I_{n_1} & 0_{n_1 \times n_2}\\
0_{(2n-k) \times n_1} &0_{(2n-k) \times n_2} \\
0_{n_2 \times n_1} & I_{n_2}
\end{pmatrix},
\]
respectively, where $I_s$ denotes an identity matrix of dimension $s \times s$ and $0_{s \times t}$ a zero matrix of dimension $s \times t$ for $s,t \in N$. Consider the linear symplectic transformation $M$ defined by
\[
M : = \begin{pmatrix}
I_n & I_n \\
0_{n\times n} & I_n
\end{pmatrix}.
\]
Applying a further symplectic change of coordinates using $M^{-1}$ on $T_z Z$, the linear spaces $T_z  X$ and $T_z  \Lambda$ are represented by the column-span of the following two matrices
\[
\begin{pmatrix}
\phantom{AA}\\
\phantom{AA} & I_n & \phantom{AA} \\
\phantom{AA}\\
& 0_{n\times n}\\
\phantom{AA}
\end{pmatrix}, 
\quad
\begin{pmatrix}
I_{n_1} & 0_{n_1 \times n_2}\\
0_{(n-k) \times n_1} &0_{(n-k) \times n_2} \\
0_{n_2 \times n_1} & I_{n_2}\\
0_{(n-n_2) \times n_1} & 0_{(n-n_2)\times n_2}\\
0_{n_2 \times n_1} & I_{n_2}
\end{pmatrix},
\]
respectively.
The linear symplectic change of coordinates $M^{-1} \circ L$ can be applied to the coordinates $q^1,\ldots,q^n,p_1,\ldots,p_n$ which will lead to a change of coordinates $M^{-1} \circ L$ on
$\left.\frac{\p }{\p q^1}\right|_z,
\ldots,
\left.\frac{\p }{\p q^n}\right|_z,
\left.\frac{\p }{\p p_1}\right|_z,
\ldots,
\left.\frac{\p }{\p p_n}\right|_z$.
Let us shrink $Z$ and all submanifolds to the considered coordinate neighbourhood.
Since a neighbourhood of $0 \in T_z  X$ is identified with $ X$, the space $ X$ is given as $\{(q,p)| p=0\}$. Now we can define $\pi \colon Z \to  X$ to be the projection $\pi(q,p) = q$ which, after shrinking $Z$ further around $z$, if necessary, fulfills all claimed properties.
\end{proof}

%
%


\begin{definition}[(fully reduced) generating function]
Let $(X,\Lambda,z)$ be a Lagrangian contact problem in $Z$. Consider a cotangent bundle structure such that $\Lambda$ is given as the image of the section $\d \phi$ locally around $z\in Z$ with $\phi(z)=0$. We refer to $\phi$, to the germ of $\phi$ at $z$, and to (the germ at 0 of) a coordinate expression of $\phi$ with centred coordinates at $z$ as a {\em generating function of $(X,\Lambda,z)$}. By the parametric Morse Lemma (\cref{lem:ParametricMorse}) there exist coordinates $x=(\overline x, \underline x)=((x^1,\ldots,x^k),(x^{k+1},\ldots,x^n))$ on $X$ centred at $z$ such that
\[
\phi (x) = f(\overline x)+\underline x^{\top} B \underline x
\]
for a smooth function germ $f$ with vanishing 2-jet and an invertible matrix $B$. The function germ $f$ is called a {\em fully reduced generating function of $(X,\Lambda,z)$}.
\end{definition}

Now \cref{prop:IndofCotStructure} can be phrased as follows.

\begin{corollary}\label{cor:GenFunWellDefined}
Generating functions of Lagrangian contact problems are defined up to stably right equivalence.
\end{corollary}

Using \cref{thm:CanellationLaw} we obtain the following corollary.

\begin{corollary}\label{cor:RedGenFunWellDefined}
Fully reduced generating functions of Lagrangian contact problems are defined up to right equivalence.
\end{corollary}


In the following we will prove a reverse direction to the above corollaries and show that Lagrangian contact problems with stably right equivalent generating functions are contact equivalent. In other words, the stably right equivalence class of the generating function of a contact problem encodes geometric information about the local intersection problem. We prepare the proof with two lemmas.

\begin{lemma}\label{lem:CotLift}
Let $\psi \colon X \to X'$ denote a diffeomorphism and $\Psi \colon T^\ast X' \to T^\ast X$ its cotangent lifted map\footnote{See footnote \ref{ft:CotLift}.}. Consider smooth functions $\phi \colon X \to \R$ and $\phi' \colon X' \to \R$ and define $\Lambda = \d \phi(X)$ and $\Lambda' = \d \phi'(X')$. If $\Psi(\Lambda') = \Lambda$ then $\phi = \phi' \circ \psi + \mathrm{const}$.
\end{lemma}

\begin{proof}
For $x' \in X'$ the element $\Psi(\d \phi'|_{x'}) \in T^\ast X$ has the base point $\psi^{-1}(x')$. As $\Psi(\d \phi'|_{x'}) \in \Lambda$, we must have
\[
\d \phi|_{\psi^{-1}(x')} = \Psi(\d \phi'|_{x'}) = \d \phi'|_{x'} \circ \d \psi|_{\psi^{-1}(x')} = \d (\phi' \circ \psi)|_{\psi^{-1}(x')}.
\]
It follows that $\phi = \phi' \circ \psi + \mathrm{const}$.
\end{proof}

\begin{lemma}\label{lem:ChangeCoords}
Consider a symplectic diffeomorphism $\Phi \colon Z \to Z'$. Any generating function of a Lagrangian contact problem $(X,\Lambda,z)$ in $Z$ is right equivalent to any generating function of the Lagrangian contact problem $(\Phi(X),\Phi(\Lambda),\Phi(z))$ in $Z'$ when shrinking all involved manifolds to intersections with open neighbourhoods of $z \in Z$ or $\Phi(z)\in Z'$, if necessary.
\end{lemma}


\begin{proof}
Let $\phi$ be a generating function for $(X,\Lambda,z)$. There exists a contangent bundle structure defined by a 1-form $\lambda$ such that $d\phi(X)=\Lambda$ (shrinking all involved manifolds, if necessary). The 1-form $\lambda' = (\Phi^{-1})^\ast \lambda$ defines a cotangent bundle structure on $Z'$ over $\Phi(X)$ around $\Phi(z)$.
Consider $\phi' \colon X' \to \R$ with $\phi'(\Phi(z))=0$ and $\d \phi'(\Phi(X)) = \Phi(\Lambda)$.
Since $\lambda' = (\Phi^{-1})^\ast \lambda$, $\Phi$ must coincide with the cotangent lift of the map $(\Phi|_X)^{-1} \colon \Phi(X) \to X$ \cite[\S 6.3]{Marsden1999}. By \cref{lem:CotLift} we have $\phi = \phi' \circ \Phi|_X$, whereas the constant vanishes due to the normalisation of $\phi$ and $\phi'$. An application of \cref{prop:IndofCotStructure} completes the proof.
\end{proof}

We can now extend \cref{prop:TouchContactifRequiv} to non-tangential Lagrangian contact problems.

\begin{theorem}\label{prop:CoEqvStabEqv}
Two Lagrangian contact problems $(X,\Lambda,z)$ in $Z$ and $(X',\Lambda',z')$ in $Z'$ with $\dim Z = \dim Z'$ are contact equivalent if and only if their generating functions are stably right equivalent.
\end{theorem}

\begin{proof}
In the following we may shrink $Z$, $Z'$ and all embedded manifolds to their intersections with open neighbourhoods of $z$ or $z'$ repeatedly without mentioning.
By \cref{lem:ChangeCoords}, any two generating functions of two contact equivalent contact problems are stably right equivalent.

To prove the reverse direction, it is sufficient to consider $Z'=Z$, $X'=X$ and $z'=z$. This is justified by the local nature of the problem and \cref{lem:ChangeCoords}.
%
%
Let $\phi$ be a generating function of $(X,\Lambda,z)$ and $\phi'$ of $(X,\Lambda',z)$ and assume that the function germs $\phi$ and $\phi'$ are stably right equivalent.
By \cref{prop:IndofCotStructure}, without loss of generality we may assume that $\phi$, $\phi'$ refer to the same cotangent bundle structure $\pi \colon Z \to X$.

By the parametric Morse Lemma (\cref{lem:ParametricMorse}) there exist coordinates $(\overline x, \underline x)=((x^1,\ldots,x^k),(x^{k+1},\ldots,x^n))$ and $(\overline {x'}, \underline {x'})=(({x'}^1,\ldots,{x'}^k),({x'}^{k+1},\ldots,{x'}^n))$ on $X$ centred at $z$, function germs $f$, $f'$ with vanishing 2-jets at 0 and invertible, symmetric matrices $B$ and $B'$ such that
\begin{align*}
\phi(\overline {x}, \underline {x})   &= f(\overline {x})  + {\underline {x}}^{\top} B \underline {x}\\
\phi'(\overline {x'}, \underline {x'}) &= f'(\overline {x'}) + {\underline {x'}}^{\top} B' \underline {x'}.
\end{align*}

Let $D$ be an invertible, symmetric, $n-k$-dimensional matrix such that the matrix $-B^{-1} + D B'D$ is invertible. 
Define the functions
\begin{align*}
\phi^{(0)}(\overline x,\underline x) &=f'(\overline {x}) + {\underline {x}}^{\top} B' \underline {x}\\
\phi^{(1)}(\overline x,\underline x) &= f(\overline {x})  + {\underline {x}}^{\top} B' \underline {x}\\
\phi^{(2)}(\overline x,\underline x) &=f(\overline {x})  + {\underline {x}}^{\top} BDB'DB \underline {x}.
\end{align*}
We now show that we can map $\Lambda' = \d \phi'(X)$ to $\d \phi^{(0)}(X)$, then to $\d \phi^{(1)}(X)$, then to $\d \phi^{(2)}(X)$, and finally to $\Lambda = \d \phi(X)$ with symplectic diffeomorphisms which fix $X$ and $z$.

\begin{itemize}
\item
The change of coordinates between $(\overline {x}, \underline {x})$ and $(\overline {x'}, \underline {x'})$ induces a diffeomorphism $\chi$ locally defined around $z \in X$ with $\chi(z)=z$.
The cotangent lift of $\chi$ fixes $z$ and provides a symplectomorphism $\Phi^{(0)}$ between $\Lambda'$ and $\d \phi^{(0)}(X)$ on an open neighbourhood of $z$ in $Z$. 

\item
Since $\phi$ and $\phi'$ are stably right equivalent, there exists a right equivalence $r$ such that $f = f'\circ r$ (\cref{thm:CanellationLaw}).
Define $\tilde r (\overline {x}, \underline {x}) = (r(\overline {x}), \underline {x})$ and denote the cotangent lift of $\tilde r$ by $\Phi^{(1)}$. Now $\Phi^{(1)}$ maps $\d \phi^{(0)}(X)$ to a manifold which coincides with $\d \phi^{(1)}(X)$ on an open neighbourhood of $z$: indeed, denote $\overline \pi(x) = \overline x$, $\underline \pi(x) = \underline x$. There exist appropriate choices of open neighbourhoods $V' \subset Z$ of $z$, $\tilde V \subset X$ of $z$, and $U,U' \subset \R^n$ of 0 such that 
\begin{align*}
\Phi^{(1)}(\Lambda'\cap V')
&=\Phi^{(1)}\left(\left\{ \d (f' \circ \overline \pi + \underline \pi^{\top}B' \underline \pi )|_{(\overline {x}, \underline {x})} \, | \, (\overline {x}, \underline {x}) \in U'\right\}\right)\\
&=\left(\left\{ \d (f \circ \overline \pi + \underline \pi^{\top} B' \underline \pi )|_{(r^{-1}(\overline {x}), \underline {x})} \, | \, (\overline {x}, \underline {x}) \in U'\right\}\right)\\
&=\left(\left\{ \d (f \circ \overline \pi + \underline \pi^{\top} B' \underline \pi )|_{(\overline {x}, \underline {x})} \, | \, (\overline {x}, \underline {x}) \in U\right\}\right)\\
&=\d \phi^{(1)}(X\cap \tilde V).
\end{align*}

\item
Denote the cotangent lift of the map $(\overline {x}, \underline {x}) \mapsto (\overline {x}, DB\underline {x})$ by $\Phi^{(2)}$. 
The symplectomorphism $\Phi^{(2)} \circ \Phi^{(1)}$ maps $\Lambda'$ to a manifold which coincides with $\d \phi^{(2)}(X)$ on a neighbourhood of $z$.

\item
It remains to show that there exists a symplectomorphism mapping $\d \phi^{(2)}(X)$ to $\Lambda$ which fixes $X$ and $z$.
Let $\lambda$ denote the canonical 1-form of the cotangent bundle structure for which $\pi \colon Z \to X$ is the projection. We define another cotangent bundle structure over $X$ by setting its canonical 1-form to $\lambda' = \lambda + \d H$ with
\[
H(x,\xi) 
= \xi^{\top}\begin{pmatrix}
0&0\\
0&\frac 14 (-B^{-1}+DB'D)
\end{pmatrix}\xi.
\]
The manifold $\Lambda$ is graphical with respect to the cotangent bundle structure defined by $\lambda'$ by the choice of $D$.
Let $\phi^{\lambda'}$ denote the generating function of $\Lambda$ with respect to the new structure $\lambda'$.
Applying \cref{lem:CompareCotStr} we get
\[
(\phi^{\lambda'} \circ k)(\overline {x},{\underline {x}}) =  f(\overline {x})  + {\underline {x}}^{\top} BDB'DB \underline {x}
\]
for a diffeomorphism $k$ on $X$. The cotangent lifted map $K$ of $k$ via the $\lambda$-structure maps $\d \phi^{\lambda'}|_x$ to $\d (\phi^{\lambda'} \circ k)|_{k^{-1}(x)} = \d (\phi^{(2)})|_{k^{-1}(x)}$ for each $x \in X$. Therefore,
\[
\d \phi^{\lambda'}(X) \quad \text{ and } \quad (K^{-1}\circ \Phi^{(2)} \circ \Phi^{(1)})(\Lambda')
\]
coincide near $z$. Locally around $z$ the symplectormorphism relating the cotangent bundle structures $\lambda$ and $\lambda'$ maps $\d \phi^{\lambda'}(X)$ to $\Lambda$ and fixes $X$ and $z$.
\end{itemize}

It follows that $(X,\Lambda,z)$ and $(X,\Lambda',z)$ are contact-equivalent.
\end{proof}

\begin{corollary}
The local algebras\footnote{See \cite{Arnold2012,lu1976singularity,Wassermann1974} for a definition and their significance for catastrophe theory.} of generating functions of contact equivalent Lagrangian contact problems are isomorphic.
\end{corollary}

Rather than using $X$ as a zero section for a cotangent bundle structure to describe the Lagrangian contact problem $(X,\Lambda,z)$ in $Z$ we can use any other Lagrangian submanifold $L \subset Z$ as a reference manifold as explained in the following proposition. This gives us some flexibility when computing generating functions which we will exploit in a parameter-dependent setting.

\begin{proposition}\label{lem:OtherRefMfld}
Let $(X,\Lambda,z)$ be a Lagrangian contact problem in the symplectic manifold $Z=T^\ast L$. Assume that $X$ and $\Lambda$ are graphical and given as the images of $\d \phi^X$ and $\d \phi^\Lambda$, for $\phi^X,\phi^\Lambda \colon L \to \R$ with $\phi^X(\pi(z))=\phi^\Lambda(\pi(z))$, where $\pi \colon Z \to L$ is the bundle projection.
The function
\begin{equation}\label{eq:MapSProp}\phi = \phi^\Lambda - \phi^X\end{equation}
expressed in local coordinates on $L$ centred at $\pi(z)$ is stably right equivalent to any generating function of $(X,\Lambda,z)$.
\end{proposition}

\begin{proof}
Denote the canonical 1-form on $Z=T^\ast L$ by $\lambda$ and the cotangent bundle projection by $\pi \colon Z \to L$. Consider the fibre preserving diffeomorphism
\[
\chi \colon Z \to Z, \quad z \mapsto z - \d \phi^X|_{\pi(z)}.
\]
We have
\[
\lambda' \coloneqq \chi^\ast \lambda = \lambda - \d (\phi^X \circ \pi).
\]
The 1-form $\lambda'$ on $Z$ is a primitive of the symplectic structure $\d \lambda$ and vanishes exactly at the points on $X$. By \cref{rem:cobundleStructure1Form}, $\lambda'$ corresponds to a cotangent bundle structure on $Z$ over $X$ with the same Lagrangian fibres as the original structure. Let $\pi' \colon Z \to X$ denote the cotangent bundle projection.
In the following we may localise the problem by intersecting $L$ with open neighbourhoods of $\pi(z)$ in $L$ and shrink $Z = T^\ast L$ and all embedded submanifolds accordingly repeatedly without mentioning. 
A generating function for $(X,\Lambda,z)$ with respect to the structure $\lambda'$ can be obtained as follows: let $S' \colon \Lambda \to \R$ be a primitive of the closed 1-form $\iota^\ast_\Lambda \lambda'$ around $z$ with $S'(z)=0$, where $\iota_\Lambda \colon \Lambda \hookrightarrow Z$ is the inclusion.
A generating function for $(X,\Lambda,z)$ with respect to $\lambda'$ is given as $\phi' = S' \circ (\pi'\circ \iota_\Lambda)^{-1} \colon X \to \R$. An expression of $\phi'$ in local coordinates on $X$ centred at $z$ is, therefore, right equivalent to an expression of $S'$ in local coordinates on $\Lambda$ centred at $z$.

To show that a coordinate expression of $\phi$ centred at $\pi(z)\in L$ and a coordinate expression of $\phi'$ centred at $z$ are right equivalent, it suffices to verify that the pullback of $\d S'$ to $L$ via $\d \phi^\Lambda \colon L \to \Lambda$ coincides with $\d \phi$. Here and in the following calculation we neglect to differentiate between $\d \phi^\Lambda \colon L \to \Lambda$ and $\d \phi^\Lambda \colon L \to Z$. Indeed, 
\begin{align*}
(\d \phi^\Lambda)^\ast (\d S')
&= (\d \phi^\Lambda)^\ast (\iota^\ast_\Lambda \lambda')
= (\d \phi^\Lambda)^\ast ( \chi^\ast \lambda)
= (\chi \circ \d \phi^\Lambda)^\ast \lambda\\
&= (\d \phi^\Lambda - \d \phi^X|_{\pi \circ \d \phi^\Lambda})^\ast \lambda
= (\d \phi^\Lambda - \d \phi^X)^\ast \lambda
= \d \phi^\ast \lambda
=\d \phi.
\end{align*}

Since generating functions are well-defined up to stably right equivalence by \cref{cor:GenFunWellDefined}, this shows that $\phi$ expressed in local coordinates centred at $\pi(z)$ is stably right equivalent to any generating function of $(X,\Lambda,z)$.
\end{proof}

\section{Parameter dependent Lagrangian contact problems}\label{sec:ParamLCP}

Let us extend our analysis of Lagrangian contact problems to parameter dependent problems. This will help us to deepen the relation of Lagrangian contact problems with catastrophe theory.

A parameter dependent version of \cref{lem:SwitchOn} holds true.

\begin{lemma}\label{lem:SwitchOnParameter}
On $\R^n = \overline X \oplus \underline X$ consider coordinates $x=(\overline x, \underline x)$ with $\overline x = (x^1,\ldots,x^k)$ and $\underline x =(x^{k+1},\ldots,x^n)$, a smooth family of nondegenerate symmetric matrices $Q_\mu \in \R^{(n-k) \times (n-k)}$, a smooth family of functions $g_\mu$ defined on an open neighbourhood of $0 \in \R^k$ such that the 2-jet of $g_0$ vanishes at 0. Here $\mu \in I$ is the family parameter and $I$ is an open neighbourhood of $0\in \R^l$.
Moreover, consider a smooth family of matrix valued functions $\mathcal H_\mu \colon \R^n \to \mathrm{Sym}(n)$ with 
\[
\mathcal H_\mu(x) = (h_\mu^{ij}(x))_{i,j=1,\ldots,n}
=
\begin{pmatrix}
H^{11}_\mu(x) & H_\mu^{12}(x) \\
{H_\mu^{12}}^{\top}(x)&0
\end{pmatrix} \in \mathrm{Sym}(n)  \subset \R^{n \times n}.
\]
for $H^{11}_\mu(x)\in \mathrm{Sym}(k)$ and $H_\mu^{12}(x) \in \R^{k \times (n-k)}$. For $t \in \R$ let
\begin{align*}
\psi(\mu,x)    & = g_\mu(\overline x) + \underline x^{\top} Q_\mu \underline x \\
\psi_{t} (\mu,x) & = g_\mu(\overline x) + \underline x^{\top} Q_\mu \underline x + t  \nabla \psi(\mu,x)^{\top} \mathcal H_\mu(x) \nabla \psi(\mu,x).
\end{align*}
Then $\psi_t$ is right equivalent to $\phi$ around $(\mu,x)=(0,0)$. The right equivalence is fibred, i.e.\ of the form $(\mu,x)\mapsto (\mu,r_\mu(x))$, and fixes $(\mu,x)=(0,0)$.
\end{lemma}

\begin{proof}
The proof is almost analogous to the proof of \cref{lem:SwitchOn} with parameter-dependent data, i.e.\ $\psi(x)$ is substituted by $\psi(\mu,x)$, $\psi_t(x)$ by $\psi_t(\mu,x)$, $\mathcal H$ by $\mathcal H_\mu$, $H^{ij}$ by $H^{ij}_\mu$, $h^{ij}$ by $h^{ij}_\mu$, $g$ by $g_\mu$, $Q$ by $Q_\mu$, $B$ by $B_\mu$, $\omega$ by $\omega_\mu$, and $f_t$ by $f_{\mu,t}$. Rather than repeating the calculations and arguments with the indicated substitutions, we point out the two adaptions that need to be made in a copy of the proof of \cref{lem:SwitchOn} with the indicated substitutions.
\begin{itemize}
\item
The 2-jet of $g_\mu$ does not necessarily vanish at $\overline x = 0$ unless $\mu=0$.
However, this is not an issue: in \eqref{eq:BlocallyInvertible} the invertibility of the matrix $B_\mu(t,0)$, which relates $\nabla_x \psi_t(\mu,0)$ and $\nabla_x \psi(\mu,0)$ follows from the invertibility of $B_0(t,0)$ for all $\mu$ close to 0 and $t \in [0,1]$, which is sufficient for the argument.
\item
Recall that $\omega_\mu(x,t)$ from \eqref{eq:ddtpsit} is the vector field whose flow map is the required right equivalence. 
Notice that $\omega_\mu(0,t)$ from \eqref{eq:ddtpsit} does not necessarily vanish if $\mu \not = 0$ and, therefore, the values $f_{\mu,t}(0)$ are only guaranteed to vanish if $\mu=0$. However, this is sufficient as the sought right equivalence needs to fix $x=0$ only at $\mu=0$. \qedhere
\end{itemize}
\end{proof}

%
%

\begin{definition}[smooth Lagrangian family]\label{def:smoothLagrangianfamily}
Let $I \subset \R^l$ be an open neighbourhood of the origin. A family $(\Lambda_\mu)_{\mu \in I}$ of Lagrangian submanifolds of a symplectic manifold $Z$ is {\em smooth} at $\mu =0$ around $z \in \Lambda_0$ if there exists an open neighbourhood  $\tilde Z$ of $z$, a cotangent bundle structure $\pi \colon \tilde Z \to \Lambda_0\cap \tilde Z$, an open neighbourhood $\tilde I \subset I$ of 0 and a smooth family of scalar valued functions $(\phi_\mu)_{\tilde I}$ such that $\Lambda_\mu \cap \tilde Z$ is the image of the section $\d \phi_\mu \colon \Lambda_0 \to \tilde Z$ for all $\mu \in \tilde I$.
\end{definition}

\Cref{def:smoothLagrangianfamily} is independent of the cotangent bundle structure $\pi \colon \tilde Z \to \Lambda_0\cap \tilde Z$ because two different structures relate by a smooth transition (\cref{lem:CompareCotStr}).

\begin{definition}[parameter-dependent Lagrangian contact problem]
Let $(X_\mu)_{\mu \in I}$ and $(\Lambda_\mu)_{\mu \in I}$ be two Lagrangian families in a symplectic manifold $Z$ that are smooth at $\mu=0$ such that $X_0 \cap \Lambda_0$ intersects in an isolated point $z_0$ and such that the set $X_\mu \cap \Lambda_\mu$ is discrete for all $\mu \in I$. Then $((X_\mu)_{\mu \in I},(\Lambda_\mu)_{\mu \in I},z_0)$ is called a {\em (parameter-dependent) Lagrangian contact problem in $Z$}.
\end{definition}

\begin{definition}[constant unfolding]
Let $\mathcal L = ((X_\mu)_{\mu \in I},(\Lambda_\mu)_{\mu \in I},z_0)$ be a Lagrangian contact problem. Let $\tilde I \subset \R^{\tilde l}$ be an open neighbourhood of 0. The Lagrangian contact problem $((X_\mu)_{(\mu, \tilde \mu) \in I \times \tilde I},(\Lambda_\mu)_{(\mu, \tilde \mu) \in I \times \tilde I},z_0)$ is called a {\em constant unfolding of $\mathcal L$}. Similarly, if $\mathcal F = (\phi_\mu)_{\mu \in I}$ is a smooth family of scalar-valued functions then $(\phi_\mu)_{(\mu, \tilde \mu) \in I \times \tilde I}$ is called a {\em constant unfolding of $\mathcal F$}.
\end{definition}


\begin{definition}[Morse-reduced form]\label{def:MorseReducedForm}
Consider an open neighbourhood $I$ of $0 \in \R^l$ and a family of scalar-valued functions $(\phi_\mu)_{\mu \in I}$ defined around $z_0 \in X$ with $\phi_0(z_0)=0$ and $\d \phi_0|_{z_0}=0$, where $X$ is an $n$-dimensional manifold.
Consider coordinates $(\overline x, \underline x) = ((x^1,\ldots,x^k),(x^{k+1},\ldots,x^n))$ centred at 0 such that $\phi_\mu$ is of the form
\[ \phi_\mu(x)=f_\mu(\overline x)+ \underline x^{\top} B \underline x \]
for a symmetric, nondegenerate matrix $B$ and a smooth family of functions $(f_\mu)_{\mu\in I}$ such that $\nabla_{\overline x} f_0(0)=0$ and $\mathrm{Hess}\, f_0(0)=0$.
Then $(f_\mu)_{\mu\in I}$ is a {\em Morse-reduced form} of $(\phi_\mu)_{\mu \in I}$.
\end{definition}

\begin{remark}
A potential parameter-dependence of the coordinates $x=(\overline x, \underline x)$ is suppressed in our notation.
\end{remark}

\begin{lemma}[existence and uniqueness of Morse-reduced forms]\label{lem:MorseRedForm}
A family $(\phi_\mu)_{\mu \in I}$ as in \cref{def:MorseReducedForm} has a Morse-reduced form $(f_\mu)_{\mu\in I}$ and $(f_\mu)_{\mu\in I}$ is locally around $(\mu,x)=(0,0)$ determined up to a right-action by a diffeomorphism of the form $K(\mu,\overline x) = (\mu, r_\mu(\overline x))$ with $K(0,0)=0$ and addition of a term $\chi(\mu)$, where $\chi$ is smooth and $\chi(0)=0$.
\end{lemma}

\begin{proof}
The lemma follows from the splitting theorem as formulated in \cite{Weinstein1971}. Compared with \cref{lem:ParametricMorse} it incorporates uniqueness properties.
Let us indicate a way to obtain the claimed existence result from \cref{lem:ParametricMorse}.
Let the dimension of the kernel of the Hessian matrix of $\phi_0$ at $z_0$ be $k$.
We find an $n-k$-dimensional submanifold $\underline X \subset X$ containing $z_0$ such that the Hessian matrix of $\phi_0|_{\underline X}$ at $z_0$ is nondegenerate.
Consider a submanifold $\overline X$ containing $z_0$ that is transversal to $\underline X$. We apply the parametric\footnote{In this context the parameters are $(\mu,\overline x)$} Morse Lemma (\cref{lem:ParametricMorse}) to
\[(\mu,x)\mapsto \phi_\mu(x)
- \underbrace{\langle \nabla_\mu \phi_\mu(0)|_{\mu=0},\mu\rangle}_{=:\chi(\mu)}-\phi_0(0)\]
with respect to the splitting $(I \oplus \overline X) \oplus \underline X$. Notice that the 1-jet of the above function with respect to the coordinates $((\mu,\overline x), \underline x)$ vanishes. A fibred change of coordinates yields local coordinates $((\mu,\overline x), \underline x)$ with $x=(\overline x, \underline x)=((x^1,\ldots,x^k),(x^{k+1},\ldots,x^n))$ on $X$ centred at $z_0$ such that $(\overline x,0)$ are coordinates on $\overline X$ and $(0,\underline x)$ are coordinates on $\underline X$ and
\[
\phi_\mu(x) = \tilde f_\mu(\overline x)+\phi_0(0)+\chi(\mu)+\underline x^{\top} B \underline x =: f_\mu(\overline x)+\underline x^{\top} B \underline x
\]
The function $\overline x \mapsto f_0(\overline x)-\phi_0(0)$ has a vanishing 2-jet at $\overline x=0$ and $B = \mathrm{Hess}\, \phi_0|_{\underline X}$ is invertible.
%
\end{proof}


\begin{proposition}\label{prop:IndofCotStructureParameter}
Let $(X,(\Lambda_\mu)_{\mu \in I},z_0)$ be a parameter dependent Lagrangian contact problem in $Z$. Let $I\subset \R^k$ denote an open neighbourhood of $0$.
Consider smooth families of cotangent bundle structures such that $\Lambda_\mu$ is given as the image of the section $\d \phi_\mu^\alpha$ and $\d \phi_\mu^\beta$ locally around $z_0\in Z$ with $\phi_0^\alpha(z_0)=0=\phi_0^\beta(z_0)$, respectively.
Then $(\phi^\alpha_\mu)_\mu$ and $(\phi^\beta_\mu)_\mu$ admit the same Morse-reduced forms up to an addition of a smooth function $\chi(\mu)$ with $\chi(0)=0$.
\end{proposition}

\begin{proof}
The following proof is a parameter-dependent version of the proof of \cref{prop:IndofCotStructure} using \cref{lem:SwitchOnParameter} instead of \cref{lem:SwitchOn}.

Let $k$ be the dimension of the kernel of the Hessian matrix of $\phi_0^\alpha$ at $z_0$ and let $n$ be the dimension of $X$.
By \cref{lem:MorseRedForm} there exist a local coordinate system $x=(\overline x, \underline x)=((x^1,\ldots,x^k),(x^{k+1},\ldots,x^n))$ on $X$ centred at $z_0$ such that
\[
\phi^\alpha_\mu(x) =  f_\mu(\overline x)+\underline x^{\top} B \underline x,
\]
the function $\overline x \mapsto f_0(\overline x)$ has a vanishing 2-jet at $\overline x=0$, and $B = \mathrm{Hess}\, \phi^\alpha_0|_{\underline X}$ is invertible.
A fibre-wise application of \cref{lem:CompareCotStr} yields
\begin{equation}\label{eq:ApplyCompCotParam}
\phi_\mu^\beta \circ k^\mu_{\alpha \beta} = \phi_\mu^\alpha + H_\mu(x,\nabla \phi_\mu^\alpha) +\chi(\mu) ,
\end{equation}
for a smooth function $\chi$ with $\chi(0)=0$ and 
\[
H_\mu(x,\xi) = \sum_{i,j =1}^n h^\mu_{ij}(x,\xi) \xi_i \xi_j
= \xi^{\top}\begin{pmatrix}
H_\mu^{11}(x,\xi)&H_\mu^{12}(x,\xi)\\
H_\mu^{12}(x,\xi)^{\top}&H_\mu^{22}(x,\xi)
\end{pmatrix}\xi
\]
with $H_\mu^{11}(x,\xi)\in \mathrm{Sym}(k)$, $H_\mu^{12}(x,\xi) \in \R^{k \times (n-k)}$, $H_\mu^{22}(x,\xi)\in \mathrm{Sym}(n-k)$, and with $k^\mu_{\alpha\beta} = \pi^\beta \circ (\pi^\alpha|_{\Lambda_\mu})^{-1}$.
For $i,j \in \{1,2\}$ define
\[
\mathcal H_\mu(x) 
=\mathcal H_\mu(\overline x, \underline x) 
= H_\mu(x,\nabla \phi_\mu^\alpha(x)),
\quad
\mathcal H_\mu^{ij}(x) 
= \mathcal H_\mu^{ij}(\overline x, \underline x) 
= H_\mu^{ij}(x,\nabla \phi_\mu^\alpha(x)).
\]
Define
\[
B'_\mu(x) = B + 4B \mathcal H_\mu^{22}(x)B.
\]

We calculate
\begin{equation}\label{eq:phibetacoordsParam}
\begin{split}
\phi_\mu^\beta(k^\mu_{\alpha \beta}(\overline x, \underline x)) 
&\stackrel{\eqref{eq:ApplyCompCotParam}}= \phi_\mu^\alpha(\overline x, \underline x) + \mathcal H_\mu(x)+ \chi(\mu)\\
&\;\;\;=  f_\mu(\overline x)+\underline x^{\top} B_\mu'(x) \underline x 
+ \nabla_{\overline x}f_\mu(\overline x)^{\top} \mathcal H_\mu^{11}(x) \nabla_{\overline x} f_\mu(\overline x) \\
&\;\;\;\phantom{=}+ 4 \underline x^{\top} B \mathcal H_\mu^{12}(x)^{\top} \nabla_{\overline x} f_\mu(\overline x) 
+ \chi(\mu). 
\end{split}
\end{equation}

We calculate the Hessian matrix of $\phi_0^\beta \circ k^0_{\alpha \beta}$ at $(\overline x, \underline x)=(0,0)$ using \eqref{eq:phibetacoordsParam} and obtain
\[
\Hess(\phi_0^\beta \circ k_{\alpha \beta})(0,0)
=\begin{pmatrix}
0&0\\0& 2B_0'(0)
\end{pmatrix}.
\]
Since $\Lambda_0$ is graphical in both cotangent bundle structures, $B'_0(0)$
must be invertible by a dimension argument.
For $(\mu,x)$ in a sufficiently small neighbourhood of $(0,0)$ the matrices $B_\mu'(x)$ are invertible and have constant signature.
By Sylvester's law of inertia, there exists a smooth family of invertible matrices $A_\mu(x)$ such that
\begin{equation}\label{eq:SylvesterBprimeParam}
A_\mu^{-\top}(x) B'_\mu(x) A_\mu^{-1}(x) = B'_\mu(0)
\end{equation}
for all $(\mu,x)$ near $(0,0)$. Consider
\[
r_\mu(\overline x,\underline x)=(\overline x, A_\mu(x) \underline x).
\]
For all $\mu$ near 0 the map $r_\mu$ fixes $x=0$ and is a diffeomorphism on a neighbourhood of $x=0$.
Let us define
\begin{equation}\label{eq:deftildeHParam}
\begin{split}
\tilde {\mathcal H}_\mu^{11}(x) &\coloneqq \mathcal H_\mu^{11}(r_\mu^{-1}(x))\\
\tilde {\mathcal H}_\mu^{12}(x) &\coloneqq \mathcal H_\mu^{12}(r_\mu^{-1}(x)) B A^{-1}_\mu(r_\mu^{-1}(x)) B_\mu'(0)^{-1}.
\end{split}
\end{equation}
By \cref{lem:SwitchOnParameter} the function
\[\phi^\alpha_1 \coloneqq f_\mu(\overline x) + \underline x^{\top} B'_\mu(0) \underline x\]
is right equivalent to
\begin{align*}
\phi^{\alpha}_2(\mu,x) &\coloneqq  f_\mu(\overline x)+\underline x^{\top} B'_\mu(0) \underline x\\
&\phantom{=}+ \nabla_{\overline x}f_\mu(\overline x)^{\top} \tilde {\mathcal H}_\mu^{11}(x) \nabla_{\overline x} f_\mu(\overline x)\\ 
&\phantom{=}+ 4 \underline x^{\top} B'_\mu(0) \tilde {\mathcal H}_\mu^{12}(x)^{\top} \nabla_{\overline x} f_\mu(\overline x). 
\end{align*}

We have
\begin{align*}
(\phi^{\alpha}_2\circ r_\mu)(x)
&=  f_\mu(\overline x)+\underline x^{\top} A_\mu(x)^{\top} B'_\mu(0) A_\mu(x)\underline x\\
&\phantom{=}+ \nabla_{\overline x}f_\mu(\overline x)^{\top} \tilde {\mathcal H}_\mu^{11}(r_\mu(x)) \nabla_{\overline x} f_\mu(\overline x)\\ 
&\phantom{=}+ 4 \underline x^{\top} A_\mu(x)^{\top} B_\mu'(0) \tilde {\mathcal H}_\mu^{12}(r_\mu(x))^{\top} \nabla_{\overline x} f_\mu(\overline x).
\end{align*}
Comparing the above with \eqref{eq:phibetacoordsParam} shows that the function $\phi^{\alpha}_2\circ r_\mu$ coincides with $\phi_\mu^\beta \circ k^\mu_{\alpha \beta} - \chi(\mu)$. This proves the claim.
\end{proof}

\begin{definition}[stably right equivalence as unfoldings]\label{def:StabEqUnfolding}
Consider open neighbourhoods $I$ of $0 \in \R^l$ and $I'$ of $0 \in \R^{l'}$ and two families of scalar-valued functions $(\phi_\mu)_{\mu \in I}$ and $(\phi'_{\mu'})_{\mu' \in I'}$ defined around $0 \in \R^{n}$ or $0 \in \R^{n'}$, respectively. Assume that $\phi_0(0)=0$ and $\phi'_0(0)=0$.
If \begin{itemize}
\item
there exists an extension of the parameter spaces $I$ to $I \times \R^{L-l}$ and $I'$ to $I' \times \R^{L-l'}$ for some $L \in \N$, $L \ge \max(l,l')$ where the additional parameters enter trivially in the function families (constant unfolding),
\item
there exist reparametrisations of both parameter spaces fixing $0 \in I \times \R^{L-l}$ and $0 \in I' \times \R^{L-l'}$, respectively,
\item
and for an open neighbourhood $\hat I \subset (I \times \R^{L-l}) \cap (I' \times \R^{L-l'})$ of 0 there exists a smooth function $\chi \colon \hat I \to \R$ 
\end{itemize}
such that the families $(\phi_{\hat \mu} + \chi(\hat \mu))_{\hat \mu \in \hat I}$ and $(\phi_{\hat \mu})_{\hat \mu \in \hat I}$ admit the same Morse-reduced forms then $(\phi_\mu)_{\mu \in I}$ and $(\phi'_{\mu'})_{\mu' \in I'}$ are {\em stably right equivalent as unfoldings}.
\end{definition}

\begin{remark}
\Cref{def:StabEqUnfolding} corresponds to the equivalence relation {\em reduction} in definition 5.3 (p.124) of Wassermann's dissertation \cite{Wassermann1974}. The necessary notions are developed in \cite[\S 4, \S 5]{Wassermann1974}. Here, however, we use right equivalence where the reference uses right-left equivalence.
\end{remark}

\begin{proposition}\label{prop:FullParameterContact}
Let $((X_\mu)_{\mu \in I},(\Lambda_\mu)_{\mu \in I},z_0)$ be a parameter-dependent Lagrangian contact problem in $Z$.
Consider two cotangent bundle structures on $Z$ over $X_0$ locally around $z_0 \in X_0 \cap \Lambda_0$ such that for each $\mu \in I$ near 0 after shrinking all involved manifolds around $z_0$, if necessary,
the submanifold $\Lambda_\mu$ is the image of the section $\d \phi_\mu$ and $X_\mu$ the image of the section $\d \psi_\mu$ with respect to the first cotangent bundle structure
and $\Lambda_\mu$ is the image of the section $\d \phi'_\mu$ 
and $X_\mu$ is the image of the section $\d \psi'_\mu$ 
with respect to the second cotangent bundle structure such that $\phi_0,\psi_0,\phi'_0,\psi'_0$ vanish at $z_0$.
Then the families $(\rho_\mu)_\mu = (\phi_\mu - \psi_\mu)_\mu$ and $(\rho'_\mu)_\mu = (\phi'_\mu - \psi'_\mu)_\mu$ are stably right equivalent as unfoldings.
\end{proposition}

\begin{proof}
We modify the first cotangent bundle structure using the fibre-preserving symplectic diffeomorphism $\xi \mapsto \xi-\d \psi_\mu|_{\pi(\xi)}$ and the second cotangent bundle structure by $\xi \mapsto \xi-\d \psi'_\mu|_{\pi'(\xi)}$ fibre-wise. In the updated structures all $X_\mu$ are zero-sections and $\Lambda_\mu$ is given as the image of the section $\d(\phi_\mu - \psi_\mu) = \d \rho_\mu$ with respect to the first structure and as the image of the section $\d(\phi'_\mu - \psi'_\mu) = \d \rho'_\mu$ with respect to the second structure. Now the claim follows by \cref{prop:IndofCotStructureParameter}.
\end{proof}

The smooth family $(\rho_\mu)_\mu$ of functions constructed in \cref{prop:FullParameterContact} is called a {\em generating family for $((X_\mu)_{\mu \in I},(\Lambda_\mu)_{\mu \in I},z_0)$}. More precisely, we can formulate the following definition.

\begin{definition}[generating family]
Consider a parameter-dependent Lagrangian contact problem $((X_\mu)_{\mu \in I},(\Lambda_\mu)_{\mu \in I},z_0)$ and a cotangent bundle structure over $X_0$ such that for $\mu$ in a neighbourhood of 0 in $I$ the manifolds $X_\mu$, $\Lambda_\mu$ are locally around $z_0$ given as the images of $\d \psi_\mu$ and $\d \phi_\mu$, respectively, for smooth function families $(\psi_\mu)_{\mu \in I}$, $(\phi_\mu)_{\mu \in I}$ with $\psi_0(z_0) = 0 = \phi_0(z_0)$. The family $(\rho_\mu)_{\mu \in I}$ with $\rho_\mu = \phi_\mu - \psi_\mu$ is called a {\em generating family for $((X_\mu)_{\mu \in I},(\Lambda_\mu)_{\mu \in I},z_0)$}.
\end{definition}


\begin{definition}[contact equivalence of parameter-dependent Lagrangian contact problems]
Let $\hat {\mathcal L} = ((X_{\hat \mu})_{{\hat \mu} \in \hat I},(\Lambda_{\hat \mu})_{{\hat \mu} \in \hat I},z_0)$ and ${\mathcal L}' = ((X'_{\mu'})_{\mu' \in I'},(\Lambda'_{\mu'})_{\mu' \in I'},z'_0)$ be Lagrangian contact problems in $Z$ and $Z'$ where $\dim Z = \dim Z'$. The families are called {\em contact equivalent} if there exists a constant unfolding $((X_\mu)_{\mu \in I},(\Lambda_\mu)_{\mu \in I},z_0)$ of $\hat {\mathcal L}$ and a constant unfolding $((X'_{\mu})_{\mu \in I},(\Lambda'_{\mu})_{\mu \in I},z'_0)$ of ${\mathcal L}'$ such that after shrinking $Z$ to an open neighbourhood of $z_0$ and $Z'$ to an open neighbourhood of $z'_0$, there exists a smooth family of symplectomorphisms
$\Phi_\mu\colon Z \to Z'$ such that for all $\mu$ in an open neighbourhood of $0\in I$
\[
\Phi_\mu(X_\mu)=X'_{\theta(\mu)},\quad \Phi(\Lambda_\mu)=\Lambda'_{\theta(\mu)}, \quad \Phi_0(z_0)=z_0',
\]
where $\theta$ is a diffeomorphism defined around $0 \in I$ fixing $\mu=0$.
\end{definition}

We conclude the section with the following theorem which extends \cref{prop:CoEqvStabEqv} to the parameter-dependent case.

\begin{theorem}\label{thm:EQParamFamilies}
Two parameter-dependent Lagrangian contact problems in a symplectic manifold are contact equivalent if and only if their generating families are stably right equivalent as unfoldings.
\end{theorem}

\begin{proof}
Let $((X_\mu)_{\mu \in I},(\Lambda_\mu)_{\mu \in I},z_0)$ and $((X'_{\mu'})_{\mu' \in I'},(\Lambda'_{\mu'})_{\mu' \in I'},z'_0)$ be parameter dependent Lagrangian contact problems in $Z$.
Since symplectic manifolds of the same dimension are locally symplectomorphic, we can assume $z_0 = z'_0$. (Also see \cref{lem:ChangeCoords}.) Moreover, as the notion of contact equivalence of Lagrangian contact problems as well as the notion of stably equivalence as unfoldings admits the extensions of the parameter spaces via constant unfoldings, it is sufficient to proof the assertion for $I = I'$.
In the following, we will shrink the manifold $Z$ to neighbourhoods of $z_0$, and $I$ to a neighbourhood of $0$, repeatedly, without mentioning.
As seen in the proof of \cref{prop:FullParameterContact} we can reduce the problem to a problem with a constant family $X_\mu \equiv X$.

For the forward direction, assume that $(X,(\Lambda_\mu)_{\mu \in I},z_0)$ and $(X,(\Lambda'_{\mu})_{\mu \in I},z_0)$ are contact equivalent.
There exists a family of symplectomorphisms $\Phi_\mu \colon Z \to Z$ with
\[
\Phi_\mu(X) = X, \quad \Phi_\mu(\Lambda_\mu) = \Lambda'_{\theta(\mu)}, \quad \Phi_0(z_0) = z_0.
\]
for a local diffeomorphism fixing $\mu=0$.
After a reparametrisation of the second family, we can assume $\theta = \id$. Now the proof proceeds similarly to the first part of the proof of \cref{prop:CoEqvStabEqv}: consider a cotangent bundle structure over $X$ such that $\Lambda_\mu$ is given as the image of $\d \phi_\mu$ for a smooth family $(\phi_\mu)_{\mu \in I}$ with $\phi_\mu(0)=0$.
Denote the cotangent bundle projection by $\pi$. Using $\Phi_\mu$ we can construct another cotangent bundle structure over $X$ with cotangent bundle projection $\pi_\mu' = \pi \circ \Phi^{-1}_\mu$.
The map $\Phi_\mu \circ \d \phi_\mu$ is a section of $\pi_\mu'$ and maps $X$ to $\Lambda'_\mu$. Therefore, in the new structure, the family $(\Lambda'_\mu)_{\mu \in I}$ can again be represented by the family $(\phi_\mu)_{\mu \in I}$.

For the reverse direction, assume that $(X,(\Lambda_\mu)_{\mu \in I},z_0)$ and $(X,(\Lambda'_{\mu})_{\mu \in I},z_0)$ have equivalent generating families. Consider a cotangent bundle structure over $X$ such that $\Lambda_\mu$, $\Lambda'_{\mu}$ are graphical and given as
\[
\phi_\mu(X)= \Lambda_\mu, \quad \phi'_{\mu}(X)= \Lambda'_{\mu}.
\]
By \cref{lem:MorseRedForm} there exist coordinates $(\overline x, \underline x)$ and $(\overline {x'}, \underline {x'})$ centred at $z_0$ such that 
\begin{align*}
\phi_\mu(\overline x, \underline x) &= f_\mu(\overline x, \underline x) + \underline x^{\top} B \underline x\\
\phi'_{\mu}(\overline {x'}, \underline {x'}) &= f'_{\mu}(\overline {x'}, \underline {x'}) + \underline {x'}^{\top} B' \underline {x'},
\end{align*}
where the 2-jets of $f$ and $f'$ vanish at $\mu=0$ and $B$ and $B'$ are nondegenerate, symmetric matrices. We have $f_\mu(\overline x) = f'_{\theta(\mu)}(r_\mu(x))$ for a diffeomorphism $\theta$ fixing $0$ and a fibred right equivalence $r$ such that $r_0(0)=0$. After a reparametrisation of the second family, we can assume $\theta = \id$. Now the proof proceeds analogously to the second part of the proof of \cref{prop:CoEqvStabEqv}. Let us sketch the four steps to construct a family of symplectic maps around $z_0$ identifying the two contact problems.
\begin{itemize}
\item
$\Lambda'_\mu$ is mapped to a manifold that coincides with $\d \phi_\mu^{(0)}(X)$ on a neighbourhood of $z_0$ with \[\phi_\mu^{(0)}(\overline {x}, \underline {x}) = f'_{\mu}(\overline {x}, \underline {x}) + \underline {x}^{\top} B' \underline {x}\] using the cotangent lift of the change of coordinates $(\overline {x}, \underline {x}) \mapsto (\overline {x'}, \underline {x'})$.
\item
On an open neighbourhood of $z_0$ the submanifold $\d \phi_\mu^{(0)}(X)$ is mapped to $\d \phi_\mu^{(1)}(X)$ with \[\phi_\mu^{(1)}(\overline {x}, \underline {x}) = f_{\mu}(\overline {x}, \underline {x}) + \underline {x}^{\top} B' \underline {x}\] using the cotangent lift of $(\overline {x}, \underline {x}) \mapsto (r_\mu(\overline {x}), \underline {x})$.
\item
On an open neighbourhood of $z_0$ the submanifold $\d \phi_\mu^{(1)}(X)$ is mapped to $\d \phi_\mu^{(2)}(X)$ with
\[\phi_\mu^{(2)}(\overline {x}, \underline {x}) = f_{\mu}(\overline {x}, \underline {x}) + \underline {x}^{\top} BDB'DB \underline {x}\]
for a suitable choice of a symmetric, invertible matrix $D$ using the cotangent lift of $(\overline {x}, \underline {x}) \mapsto (\overline {x}, DB\underline {x})$.
\item
On an open neighbourhood of $z_0$ the submanifold $\d \phi_\mu^{(2)}(X)$ is mapped to $\tilde \Lambda_\mu$ using the cotangent lift of the fibred right equivalence obtained from \cref{lem:SwitchOnParameter} for 
\[
\mathcal H = \begin{pmatrix}
0&0\\
0& \frac 14(-B^{-1}+ DB'D)
\end{pmatrix} \in \R^{n \times n}.
\]
Let the original cotangent bundle structure be defined by the 1-form $\lambda$.
$\Lambda_\mu$ is mapped to $\tilde \Lambda_\mu$ using the symplectomorphism that identifies the cotangent bundle structure $\lambda$ with $\lambda + \d H$, where $H(x,\xi) = \xi^{\top} \mathcal H(x, \xi) \xi$ in Darboux coordinates with respect to the original structure.
\end{itemize}
%
%
%
%
The procedure constructs the required family of symplectomorphisms.
\end{proof}

\section[Conclusion and applications]{Concluding remarks and application to boundary value problems of symplectic maps}\label{sec:Conclusion}

\subsection{Stably contact equivalence}

It is now justified to extend the notion of contact equivalence of Lagrangian contact problems to {\em stably contact equivalence}. This will allow us to compare Lagrangian contact problems in symplectic manifolds of {\em different dimensions}.



\begin{definition}
\label{def:StableEQContact2}
Two Lagrangian contact problems (in symplectic manifolds of possibly different dimensions)
are {\em stably contact equivalent} if their generating functions are stably right equivalent.
Moreover, two parameter-dependent Lagrangian contact problems (in symplectic manifolds of possibly different dimensions) are {\em stably contact equivalent} if their generating families are stably right equivalent as unfoldings.
\end{definition}

\begin{proposition}
Two Lagrangian contact problems in symplectic manifolds of the same dimension are stably contact equivalent if and only if they are contact equivalent.
\end{proposition}

\begin{proof}
The statement follows by the \cref{prop:CoEqvStabEqv,thm:EQParamFamilies}.
\end{proof}

The following two theorems announced in \cref{sec:intro} now follow trivially from \cref{def:StableEQContact2}. Their geometric meaning is encoded in the \cref{prop:CoEqvStabEqv,thm:EQParamFamilies}.

\CorrSingularities*

%
%


\CorBifur*

\begin{remark}\label{rem:RelationToCatastropheTheory}
The notion of versality and stability can be translated to the setting of (parameter-dependent) Lagrangian contact problems such that classification results from catastrophe theory apply (see \cite[Part II]{Arnold2012}, for instance).
%
\end{remark}


\subsection{Boundary value problems for symplectic maps}

An application is the classification of singularities and bifurcations which occur in boundary value problems for symplectic maps \cite{bifurHampaper,PhDThesis}:
consider a smooth family of symplectic maps $\phi_\mu \colon Z \to Z'$ for $\mu \in I$, where $I \subset \R^l$ is an open neighbourhood of the origin. Let us denote the symplectic form of $Z$ by $\omega$ and the symplectic form of $Z'$ by $\omega'$.
Let $\pr \colon Z \times Z' \to Z$ and $\pr' \colon Z \times Z' \to Z'$ denote projections to the first or second component of the product. Define the symplectic form $\Omega = \pr^\ast \omega-{\pr'}^\ast \omega'$ on $Z \times Z'$. The graphs of $(\phi_\mu)_\mu$ define a smooth family $(\Lambda_\mu)_\mu$ of Lagrangian submanifolds in $Z \times Z'$.

The Lagrangian contact problems $((\Lambda_\mu)_\mu,(X_\mu)_\mu,z)$ for a smooth family $(X_\mu)_\mu$ of Lagrangian submanifolds of $Z \times Z'$ and a point $z \in Z \times Z'$ can be interpreted as a family of boundary value problems for the symplectic maps $(\phi_\mu)_\mu$, where $(X_\mu)_\mu$ represents the boundary conditions. 
The elements of the intersection $\Lambda_\mu \cap X_\mu$ correspond to solutions to the boundary value problem.


\begin{example}[periodic boundary conditions] Consider $Z'=Z$, let $(\phi_\mu)_{\mu \in I}$ be a family of symplectic maps on $Z$, let $\Lambda_\mu$ denote the graph of $\phi_\mu$ viewed as a subset of $Z\times Z$ and let $X_\mu \equiv X$ be the diagonal embedding of $Z$ into $Z\times Z$. The elements of the intersection $\Lambda_\mu \cap X$ correspond to solutions to the boundary value problem $\phi_\mu(z)=z$, $z\in Z$.
\end{example}

\begin{example}[Dirichlet-type boundary conditions]
Consider a family of Hamiltonians $(H_\mu)_{\mu \in I}$ on a cotangent bundle space $T^\ast U$ with bundle projection $\pi \colon T^\ast U \to U$. Let $q^\ast,Q^\ast\in U$ and let $\phi_\mu$ denote the time-1-map corresponding to the Hamiltonian system $(H_\mu,T^\ast U,-\d \lambda)$. 
Let $(q,p)$ denote local canonical coordinates on $T^\ast U$ around a point of interest where $Z$ is a coordinate patch such that $q^\ast\in \pi(Z)$ and $Z'$ is a coordinate patch such that $Q^\ast\in \pi(Z')$. The equation
\begin{equation}\label{eq:DirichletHamEx}
\pi(\phi_\mu(q^\ast,p)) = Q^\ast
\end{equation}
is a boundary value problem for $\phi_\mu$. This kind of boundary value problems arise as first-order formulations of parameter-dependent second-order systems of ordinary differential equations with Dirichlet boundary conditions, for instance.
Let $\Lambda_\mu$ denote the graph of $\phi_\mu\colon Z \to Z'$ viewed as a subset of $Z\times Z'$ and let $X_\mu \equiv X = \pr^{-1}(q^\ast) \times {\pr'}^{-1}(Q^\ast)$. The elements of the intersection $\Lambda_\mu \cap X$ correspond to solutions to the boundary value problem \eqref{eq:DirichletHamEx}.
\end{example}

Two parameter dependent boundary value problems for symplectic maps are {\em equivalent} at a solution if and only if their corresponding parameter-dependent Lagrangian contact problems in $Z \times Z'$ are contact equivalent.
Therefore, \cref{rem:RelationToCatastropheTheory} applies to parameter-dependent boundary value problems for symplectic maps and relates the problem to classical catastrophe theory. For this also notice that small perturbations of the corresponding catastrophe theory problems relate to perturbations of the Lagrangian contact problems which can still be interpreted as boundary value problems for symplectic maps.
%
%
Indeed, we obtain a
rigorous framework for the work done in \cite{bifurHampaper}, i.e.\ the analysis of bifurcations in boundary value problems for symplectic maps. 
In particular, \cref{thm:EQParamFamilies} fills a gap in the argumentation of proposition 2.1 in \cite{bifurHampaper}.
Implications for numerical computations of bifurcation diagrams are explained in \cite{numericalPaperShort,numericalPaper,PhDThesis}
and illustrated in \cref{fig:CatastrophiesBreak}.

\begin{figure}
	
	\includegraphics[width=0.45\textwidth]{D4minus_c.jpg}
	\includegraphics[width=0.45\textwidth]{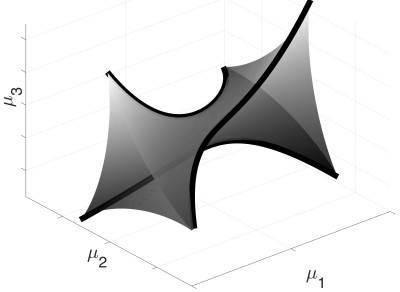}
	
	\caption{When bifurcation diagrams of parameter dependent Hamiltonian boundary value problems are computed using a numerical method, it is crucial that the symplectic flow map is approximated by an algorithmically computable flow map which is symplectic as well. The symplecticity guarantees that the computed bifurcation diagram exhibits the same stable singularities as the bifurcation diagram of the exact system. To the left we see an illustration of an elliptic umbilic bifurcation which is persistent under perturbations within the class of boundary value problems for symplectic maps. If a numerical discretisation error perturbs the problem to a boundary value problem for a non-symplectic map, the bifurcation breaks (right hand side) \cite{numericalPaperShort,numericalPaper,PhDThesis}.  }\label{fig:CatastrophiesBreak}
\end{figure}

\begin{example}
Singularities in conjugate loci can be viewed as singularities of exponential maps. Interpreting exponential maps as Lagrangian maps \cite{Janeczko1995}, the singularities can be classified via generating families up to stably $R^+$-equivalence \cite[p.304]{Arnold2012}. 
Alternatively, elements of conjugate loci can be interpreted as singularities of boundary value problems for exponential maps. Using a slightly different notion to identify the singularities, elements in the conjugate loci correspond to map germs up to right equivalence with certain symmetries \cite{obstructionPaper}.
\end{example}


\section*{Acknowledgements}
I thank Robert McLachlan for useful discussions. This research was supported by the Marsden Fund of the Royal Society Te Ap\={a}rangi and the School of Fundamental Sciences, Massey University, Manawat\={u}, New Zealand.

\printbibliography

@book{lu1976singularity,
  title={Singularity theory and an introduction to catastrophe theory},
  author={Lu, Yung-Chen},
  isbn={9783540902218},
  lccn={76048307},
  series={Universitext (1979)},
  doi={10.1007/978-1-4612-9909-7},
  year={1976},
  publisher={Springer-Verlag}
}

@Inbook{Libermann1987,
author="Libermann, Paulette and Marle, Charles-Michel",
title="Symplectic manifolds and {P}oisson manifolds",
bookTitle="Symplectic Geometry and Analytical Mechanics",
year="1987",
publisher="Springer Netherlands",
address="Dordrecht",
pages="89--184",
isbn="978-94-009-3807-6",
doi="10.1007/978-94-009-3807-6_3"
}

@article{coisoPair,
Author = {Lorand, Jonathan and Weinstein, Alan},
Title = {({C}o)isotropic Pairs in {P}oisson and Presymplectic Vector Spaces},
Year = {2015},
Eprint = {1503.00169},
Howpublished = {SIGMA 11 (2015), 072, 10 pages},
Journal={SIGMA 11},
Volume={11},
Doi = {10.3842/SIGMA.2015.072}
}

@book{fukaya2010lagrangian,
  title={{L}agrangian Intersection Floer Theory: Anomaly and Obstruction, Part I},
  author={Fukaya, Kenji},
  isbn={9780821852491},
  series={AMS/IP studies in advanced mathematics},
  year={2010},
  publisher={American Mathematical Society}
}

@Inbook{LagrangianInterTheo,
author={Eliashberg, Y. and Gromov, M.},
title={{L}agrangian intersection theory: finite-dimensional approach},
bookTitle={Geometry of Differential Equations},
pages={27--118},
isbn={9780821810941},
lccn={98204441},
series={AMS Translations Series},
year={1998},
publisher={American Mathematical Society}
}

@article{Frauenfelder2004,
    author = {Frauenfelder, Urs},
    title = "{The {Arnold-Givental} conjecture and moment {F}loer homology}",
    journal = {International Mathematics Research Notices},
    volume = {2004},
    number = {42},
    pages = {2179-2269},
    year = {2004},
    month = {01},
    issn = {1073-7928},
    doi = {10.1155/S1073792804133941},
    bsd-url = {https://doi.org/10.1155/S1073792804133941},
    bsd-eprint = {https://academic.oup.com/imrn/article-pdf/2004/42/2179/2148198/2004-42-2179.pdf}
}

@book{poston1978catastrophe,
  title={Catastrophe Theory and Its Applications},
  author={Poston, T. and Stewart, I.},
  isbn={9780486692715},
  lccn={96021795},
  series={Dover Books on Mathematics},
  year={1978},
  publisher={Dover Publications}
}

@article{bifurHampaper,
  author={McLachlan, Robert I. and Offen, Christian},
  title={Bifurcation of solutions to {H}amiltonian boundary value problems},
  journal={Nonlinearity},
  volume={31},
  number={6},
  pages={2895--2927},
  doi={10.1088/1361-6544/aab630},
  eprint={1710.09991},
  archivePrefix={arXiv},
  year={2018}
}

@article{PDEBifur,
  author={Kreusser, Lisa M. and McLachlan, Robert I. and Offen, Christian},
  title={Detection of high codimensional bifurcations in variational {PDE}s},
  journal={Nonlinearity},
  volume={33},
  number={5},
  pages={2335--2363},
  doi={10.1088/1361-6544/ab7293},
  eprint={1903.02659},
  archivePrefix={arXiv},
  bsd-url={https://doi.org/10.1088/1361-6544/ab7293},
  year={2020}
}

@article{numericalPaper,
Author = {McLachlan, Robert I. and Offen, Christian},
Title = {Preservation of bifurcations of {H}amiltonian boundary value problems under discretisation},
Year = {2020},
volume={20},
pages={1363-1400},
journal = {Foundations of Computational Mathematics (FoCM)},
Eprint = {1804.07468},
doi = {10.1007/s10208-020-09454-z},
eprint={1804.07468},
archivePrefix={arXiv}
}

@article{numericalPaperShort,
   title={Symplectic integration of boundary value problems},
   volume={81},
   ISSN={1572-9265},
   DOI={10.1007/s11075-018-0599-7},
   number={4},
   journal={Numerical Algorithms},
   publisher={Springer Science and Business Media LLC},
   author={McLachlan, Robert I. and Offen, Christian},
   year={2019},
   month={8},
   pages={1219--1233},
   eprint={1804.09042},
   archivePrefix={arXiv}
}

@article{obstructionPaper,
Author = {McLachlan, Robert I. and Offen, Christian},
Title = {Hamiltonian boundary value problems, conformal symplectic symmetries, and conjugate loci},
journal = {New Zealand Journal of Mathematics (NZJM)},
pages={83--99},
volume = {48},
Year = {2018},
url = {http://nzjm.math.auckland.ac.nz/index.php/Hamiltonian_Boundary_Value_Problems\%2C_Conformal_Symplectic_Symmetries\%2C_and_Conjugate_Loci},
eprint={1804.07479},
archivePrefix={arXiv}
}

@Article{MatherIII,
author="Mather, John N.",
title="Stability of ${C}^\infty$ mappings: {III}. {F}initely determined map-germs",
journal="Publications Math{\'e}matiques de l'Institut des Hautes {\'E}tudes Scientifiques",
year="1968",
month="12",
day="01",
volume="35",
number="1",
pages="127--156",
issn="1618-1913",
doi="10.1007/BF02698926",
bsd-url="https://doi.org/10.1007/BF02698926"
}

@article{Janeczko1995,
  title={Relative generic singularities of the exponential map},
  author={Janeczko, Stanis{\l}aw and Mostowski, Tadeusz},
  journal={Compositio Mathematica},
  volume={96},
  number={3},
  pages={345--370},
  year={1995},
  publisher={Groningen: P. Noordhoff, 1935-},
  url = {http://www.numdam.org/item?id=CM_1995__96_3_345_0}
}

@book{Arnold2012,
	doi = {10.1007/978-0-8176-8340-5},
	year = 2012,
	publisher = {Birkhäuser Boston},
	author = {Arnold, Vladimir I. and Gusein-Zade, S. M. and Varchenko, A. N.},
	title = {Singularities of Differentiable Maps, Volume 1}
}

@article{golubitsky1975contact,
  title={Contact equivalence for {L}agrangian manifolds},
  author={Golubitsky, Martin A. and Guillemin, Victor},
  journal={Advances in Mathematics},
  volume={15},
  number={3},
  pages={375--387},
  year={1975},
  doi = "10.1016/0001-8708(75)90143-7",
  publisher={Academic Press}
}

@Inbook{Marsden1999,
author="Marsden, Jerrold E. and Ratiu, Tudor S.",
title="Cotangent Bundles",
bookTitle="Introduction to Mechanics and Symmetry: A Basic Exposition of Classical Mechanical Systems",
year="1999",
publisher="Springer New York",
address="New York, NY",
pages="165--180",
isbn="978-0-387-21792-5",
doi="10.1007/978-0-387-21792-5_6",
bsd-url="https://doi.org/10.1007/978-0-387-21792-5_6"
}

@book{Arnold1992,
	doi = {10.1007/978-3-642-58124-3},
	year = 1992,
	isbn={978-3-642-58124-3},
	publisher = {Springer Berlin Heidelberg},
	author = {Arnold, Vladimir I.},
	title = {Catastrophe Theory}
}

@book{Arnold1994,
author="Arnold, Vladimir I. and Afrajmovich, V.S. and Il'yashenko, Y.S. and Shil'nikov, L.P.",
Title="Dynamical Systems V: Bifurcation Theory and Catastrophe Theory",
year="1994",
publisher="Springer Berlin Heidelberg",
address="Berlin, Heidelberg",
isbn="978-3-642-57884-7",
doi="10.1007/978-3-642-57884-7"
}

@article{Lomel2008,
	doi = {10.1088/0951-7715/21/3/007},
	year = 2008,
	month = {2},
	publisher = {{IOP} Publishing},
	volume = {21},
	number = {3},
	pages = {485--508},
	author = {Lomel\'i, H\'ector E and Meiss, James D  and Ram\'irez-Ros, Rafael},
	title = {Canonical {M}elnikov theory for diffeomorphisms},
	journal = {Nonlinearity}
}

@article{Haro2000,
	doi = {10.1088/0951-7715/13/5/304},
	year = 2000,
	month = {6},
	publisher = {{IOP} Publishing},
	volume = {13},
	number = {5},
	pages = {1483--1500},
	author = {Haro, \'Alex},
	title = {The primitive function of an exact symplectomorphism},
	journal = {Nonlinearity}
}

@book{brocker1975,
place={Cambridge}, 
series={London Mathematical Society Lecture Note Series}, 
title={Differentiable Germs and Catastrophes}, 
DOI={10.1017/CBO9781107325418}, 
bsd-url={https://doi.org/10.1017/CBO9781107325418},
publisher={Cambridge University Press}, 
author={Br{\"o}cker, Theodor}, 
year={1975}, 
collection={London Mathematical Society Lecture Note Series}
}

@book{Wassermann1974,
	doi = {10.1007/bfb0061658},
	year = 1974,
	publisher = {Springer Berlin Heidelberg},
	isbn = {978-3-540-38423-6},
	author = {Gordon Wassermann},
	title = {Stability of Unfoldings}
}

@article{Weinstein1971,
author = {Weinstein, Alan},
title = {Singularities of families of functions},
journal = {Differentialgeometrie im Grossen (Ber. Tagung, Math. Forschungsinst., Oberwolfach, 1969)},
volume = {4},
pages = {323-330},
publisher = {Ber. Math. Forschungsinst. Oberwolfach, Bibliographisches Inst. Mannheim},
year =  {1971} 
}

@book{WeinsteinBates,
	year = 1997,
	publisher = {American Mathematical Society},
	isbn = {978-0-8218-0798-9},
	author = {Bates, Sean and Weinstein, Alan},
	title = {Lectures on the Geometry of Quantization},
	volume = {8}
}

@misc{YoutubeSingularitiesAnimations,
        title = {Singularities Animations},
        year = {2019},
        organization = {Youtube},
        author = {Offen, Christian},
        howpublished = {\url{https://www.youtube.com/playlist?list=PLIp-UrijLTJ5m-3ZASHPurIkehiBuW_sO}},
	note = {Accessed 2021-05-02}
    }

@article{WeinstBvPHam,
 ISSN = {0003486X},
 author = {Weinstein, Alan},
 journal = {Annals of Mathematics},
 number = {3},
 pages = {377--410},
 publisher = {Annals of Mathematics},
 title = {Lagrangian Submanifolds and {H}amiltonian Systems},
 volume = {98},
 doi={10.2307/1970911},
 year = {1973}
}

@misc{lor2014coisotropic,
    title={Coisotropic Pairs},
    author={Lorand, Jonathan and Weinstein, Alan},
    year={2014},
    eprint={1408.5620},
    archivePrefix={arXiv},
    primaryClass={math.SG}
}

@article{Tougeron1968,
     author = {Tougeron, Jean-Claude},
     title = {Id\'eaux et fonctions diff\'erentiables. {I}},
     journal = {Annales de l'Institut Fourier},
     publisher = {Imprimerie Louis-Jean},
     address = {Gap},
     volume = {18},
     number = {1},
     year = {1968},
     pages = {177-240},
     doi = {10.5802/aif.281},
     zbl = {0188.45102},
     mrnumber = {39 \#2171},
     language = {fr}
}

@article{Weinstein1972,
     author = {Weinstein, Alan},
     title = {The invariance of {P}oincaré's generating function for canonical transformations},
     journal = {Inventiones mathematicae},
     publisher = {Springer},
     volume = {16},
     number = {3},
     year = {1972},
     pages = {202--213},
     doi = {10.1007/BF01425493}
}

@book{PhDThesis,
place={Manawat\={u}, New Zealand},
title={Analysis of {H}amiltonian boundary value problems and symplectic integration ({D}octoral {T}hesis)},
publisher={Massey University},
author={Offen, Christian},
doi = {10.13140/RG.2.2.34063.61607},
year={2020} }



\end{document}